\documentclass[reqno]{amsart}
\usepackage{amsfonts}

\newtheorem{theorem}{Theorem}
\theoremstyle{plain}

\newtheorem{corollary}{Corollary}

\newtheorem{definition}{Definition}

\newtheorem{lemma}{Lemma}

\newtheorem{proposition}{Proposition}
\newtheorem{remark}{Remark}

\DeclareMathOperator{\Div}{div}
 \numberwithin{equation}{section}
 \numberwithin{theorem}{section}
 \numberwithin{proposition}{section}
 \numberwithin{remark}{section}
 \numberwithin{definition}{section}
 \numberwithin{lemma}{section}
 \numberwithin{corollary}{section}
 \numberwithin{example}{section}
 \numberwithin{claim}{section}
\input{tcilatex}

\begin{document}
\title[Multiscale nonlocal flow ]{Multiscale nonlocal flow in a fractured
porous medium}
\author{Jean Louis Woukeng}
\address{Department of Mathematics and Computer Science, University of
Dschang, P.O. Box 67, Dschang, Cameroon}
\email{jwoukeng@yahoo.fr}
\date{}

\begin{abstract}
We study the flow generated by an incompressible viscoelastic fluid in a
fractured porous medium. The model consists of a fluid flow governed by
Stokes-Volterra equations evolving in a periodic double-porosity medium.
Using the multiscale convergence method associated to some recent tools
about the convergence of convolution sequences, we show that the equivalent
macroscopic model is of the same type as the microscopic one, but in a fixed
domain.
\end{abstract}

\maketitle

\section{Introduction}

The Stokes equations have been for a long time widely used to describe the
flow at moderate velocity of incompressible viscous fluids. However, models
of viscoelastic fluids have been proposed in the twentieth century. Some of
these models take into account the history of the flow and are not subject
to the Newtonian effects. Among these models is the one proposed by Oldroyd 
\cite{Oldroyd} in 1956, and which is commonly known as \emph{Oldroyd model}.
For details about the physical background and its mathematical modelling, we
refer e.g., to Oldroyd \cite{Oldroyd}, Joseph \cite{Joseph}, Oskolkov \cite%
{Oskolkov, Oskolkov1}, Agranovich et al. \cite{Agra} and Sobolevskii \cite%
{Sobo}.

In this work we focus on flows of an incompressible viscoelastic fluid of
Oldroyd type in a multiscale porous medium. The equations of motion arising
from that model give rise to a system of integro-differential equations of
Volterra-Stokes type (\ref{2.1}) and (\ref{2.3}) below; see e.g., \cite%
{Oskolkov, Oskolkov1}. The main motivation for our study is twofold: First,
the specific structure of the medium; second, the flow equations. Let us
clarify this below.

(1) \emph{The specific structure of the porous medium}. In 1990, Arbogast,
Douglas and Hornung \cite{ADH} developed a double porosity geometry which
has been studied since then by many researchers. Very recently in 2010 and
in 2012, Meirmanov \cite{Meir10, Meir12} showed that the model in \cite{ADH}%
\ was not \emph{"physically correct"} and then proposed the physically
correct double porosity geometry which distinguishes the pores size from the
cracks size and hence leads to \emph{reiterated} homogenization; see also 
\cite{Panfilov13} for a very recent work about the periodic multiscale
fractured porous media. Our approach of modelling the geometry of the domain
will be a blend of the methods used in \cite{Meir10, Meir12} and \cite%
{Panfilov13, SSW, SW1}. More precisely, the medium represents a system of
porous permeable blocks surrounded by an interconnected system of cracks
(fractures), the blocks being less permeable than the fractures. Moreover
the pores contained in the blocks are intertwined, and the system consisting
of both pores and fractures is interconnected. This is therefore a typical
fractured porous medium as defined earlier in \cite{Meir10, Meir12,
Panfilov13, SSW, SW1}. In such media, the blocks behave as sources of fluid
alimenting the fissures (fractures) which are characterized by substantially
higher flow rates and lower relative volume. For that reason, the average
flow in the block is always delayed with respect to the flow in the
fractures. This phenomenon is analytically characterized by the appearance
of memory in the resulting model equation, which is the justification of the
main physical expected feature of such media. To be more precise, the domain
is described as follows.

Let $N\geq 3$ be an integer. Let $Y=[0,1)^{N}=\overline{Y}_{1}\cup Y_{2}$
where $Y_{1}$ and $Y_{2}$ are two disjoint open sets representing the local
structure of the porous matrix for $Y_{1}$, and the local structure of
cracks for $Y_{2}$. We assume that $Y_{2}$ is connected and that the
boundary of $Y_{1}$ is Lipschitz continuous. We set $G_{1}=\cup _{k\in 
\mathbb{Z}^{N}}(k+Y_{1})$ and $G_{2}=\mathbb{R}^{N}\backslash \overline{G}%
_{1}$. Next, let $Z_{1}$ and $Z_{2}$ be two disjoint open subsets of $Y_{1}$
such that $Y_{1}=\overline{Z}_{1}\cup Z_{2}$, $Z_{2}$ being connected. $%
Z_{1} $ is the local structure of the skeleton while $Z_{2}$ is the one of
the pores. Set $H_{1}=\cup _{k\in \mathbb{Z}^{N}}(k+Z_{1})$ and $%
H_{2}=G_{1}\backslash \overline{H}_{1}$. $H_{2}$ is open and connected,
representing the effective pore space. The crack space $G_{2}$ is also
connected. Finally we assume that $Y_{2}$ and $Z_{2}$ have positive Lebesgue
measure.

With this in mind, let $\Omega $ be an open bounded connected subset of $%
\mathbb{R}^{N}$ with Lipschitz boundary $\partial \Omega $, and let $%
\varepsilon >0$ be a small parameter. We define the fluid domain to be the
union of both pores and cracks domains as follows. First set $\Omega
_{p}^{\varepsilon }=\Omega \cap \varepsilon ^{2}H_{2}$ (the pores domain), $%
\Omega _{c}^{\varepsilon }=\Omega \cap \varepsilon G_{2}$ (the cracks
domain), and define the fluid domain $\Omega _{2}^{\varepsilon }=\Omega
_{p}^{\varepsilon }\cup \Omega _{c}^{\varepsilon }$, and the skeleton $%
\Omega _{1}^{\varepsilon }=\Omega \backslash \overline{\Omega }%
_{2}^{\varepsilon }$. It holds that $\Omega =\Omega _{1}^{\varepsilon }\cup
\Gamma _{12}^{\varepsilon }\cup \Omega _{2}^{\varepsilon }$ (disjoint union)
where $\Gamma _{12}^{\varepsilon }=\partial \Omega _{1}^{\varepsilon }\cap
\partial \Omega _{2}^{\varepsilon }$ is the interface of $\Omega
_{1}^{\varepsilon }$ and $\Omega _{2}^{\varepsilon }$. Let $\nu _{j}$ ($%
j=1,2 $) denote the unit outward normal on $\partial \Omega
_{j}^{\varepsilon }$.

As an illustration of the analytic construction made above, see figure 1
below.

Fig.1: Fractured porous medium%

There are several ways to model fractured porous media. The classical and
most studied model named double porosity model was introduced by Barenblatt
et al. \cite{BKZ60}, and has been further developed in \cite{WR63, CS64,
ADH, HP83, Kaz69, Ode65, PS07, PSY09, SV, YPS10, JMS2}. In contrast, very
few works are devoted to multiscale media with the geometry similar to the
one considered in this paper. We may cite \cite{Meir10, Meir12, Panfilov13,
Leme, SSW, SW1}.

(2) \emph{The flow equations}. Motivated by the above-mentioned
phenomenological feature of our medium (the memory appearance), we found
necessary to study in such media, the flow of fluid having memory. That is
why we consider a generalized class of Oldroyd incompressible viscoelastic
fluids including as a special case, the classical Newtonian flow of Stokes'
type. They are modeled by a system of integro-differential equations in
which all the coefficients and memory kernels depend on both fast and slow
space-- and time-- variables. From the physical point of view, it means that
the memory effects arising by meeting an obstacle decay in the surrounding
of the next obstacle. For that simple reason, we can not expect using the
Laplace transform to perform the homogenization process since the memory
kernels depend on the fast time variable. Therefore, to achieve our goal,
instead of using the Laplace transform, we use a direct method involving
some results about the reiterated convergence of sequences defined by
convolution (see Theorem \ref{t3.5} which is new in the context). Let us
emphasize that although nonlocal in time terms can appear in some
homogenization problems (see e.g. \cite{Panfilov13, Yeh}), our approach can
handle more complicated homogenization problems with nonlocal terms in both
time and space variables; see e.g. \cite{M2AS}. As far as we know, this is
the first time that such a problem is considered in the literature.
Therefore, taking into account both the structure of the media and the model
equations, we can see that our main result is new.

To be more precise, we consider a non-stationary flow of an incompressible
viscoelastic non-Newtonian fluid governed by the Stokes system. The
viscoelastic constitutive law associated to the momentum balance, and the
continuity equations of the normal stress and velocity at the interface are
given by (for a.e. $0<t<T$, $T$ being given) 
\begin{equation}
\rho _{1}^{\varepsilon }\frac{\partial \boldsymbol{u}_{\varepsilon }}{%
\partial t}-\Div\sigma _{1}^{\varepsilon }=\rho _{1}^{\varepsilon }f_{1}%
\text{ in }\Omega _{1}^{\varepsilon }  \label{2.1}
\end{equation}%
\begin{equation}
\Div\boldsymbol{u}_{\varepsilon }=0\text{ in }\Omega _{1}^{\varepsilon }
\label{2.2}
\end{equation}%
\begin{equation}
\rho _{2}^{\varepsilon }\frac{\partial \boldsymbol{v}_{\varepsilon }}{%
\partial t}-\Div\sigma _{2}^{\varepsilon }=\rho _{2}^{\varepsilon }f_{2}%
\text{ in }\Omega _{2}^{\varepsilon }  \label{2.3}
\end{equation}%
\begin{equation}
\Div\boldsymbol{v}_{\varepsilon }=0\text{ in }\Omega _{2}^{\varepsilon }
\label{2.4}
\end{equation}%
\begin{equation}
\boldsymbol{u}_{\varepsilon }=\boldsymbol{v}_{\varepsilon }\text{ on }\Gamma
_{12}^{\varepsilon }  \label{2.5}
\end{equation}%
\begin{equation}
\sigma _{1}^{\varepsilon }\cdot \nu _{1}=\sigma _{2}^{\varepsilon }\cdot \nu
_{1}\text{ on }\Gamma _{12}^{\varepsilon }  \label{2.6}
\end{equation}%
\begin{equation}
\boldsymbol{u}_{\varepsilon }=0\text{ on }(\partial \Omega _{1}^{\varepsilon
}\cap \partial \Omega ),\ \boldsymbol{v}_{\varepsilon }=0\text{ on }%
(\partial \Omega _{2}^{\varepsilon }\cap \partial \Omega )  \label{2.7}
\end{equation}%
\begin{equation}
\boldsymbol{u}_{\varepsilon }(x,0)=\boldsymbol{u}^{0}(x)\text{ in }\Omega
_{1}^{\varepsilon }\text{ and }\boldsymbol{v}_{\varepsilon }(x,0)=%
\boldsymbol{v}^{0}(x)\text{ in }\Omega _{2}^{\varepsilon }.  \label{2.8}
\end{equation}%
In the above equations, denoting by $I$ is the identity tensor, 
\begin{equation*}
\sigma _{1}^{\varepsilon }=-p_{\varepsilon }I+A_{0}^{\varepsilon }\nabla 
\boldsymbol{u}_{\varepsilon }+\int_{0}^{t}A_{1}^{\varepsilon }(x,t-\tau
)\nabla \boldsymbol{u}_{\varepsilon }(x,\tau )d\tau
\end{equation*}%
is the stress tensor of the fluid in $\Omega _{1}^{\varepsilon }$ with
density $\rho _{1}^{\varepsilon }$, velocity $\boldsymbol{u}_{\varepsilon }$
and pressure $p_{\varepsilon }$ while 
\begin{equation*}
\sigma _{2}^{\varepsilon }=-q_{\varepsilon }I+B_{0}^{\varepsilon }\nabla 
\boldsymbol{v}_{\varepsilon }+\int_{0}^{t}B_{1}^{\varepsilon }(x,t-\tau
)\nabla \boldsymbol{v}_{\varepsilon }(x,\tau )d\tau
\end{equation*}%
is the stress tensor of the fluid in $\Omega _{2}^{\varepsilon }$ with
density $\rho _{2}^{\varepsilon }$, velocity $\boldsymbol{v}_{\varepsilon }$%
and pressure $q_{\varepsilon }$; $\rho _{1}^{\varepsilon }f_{1}$ and $\rho
_{2}^{\varepsilon }f_{2}$ are the external body forces per volume. The
functions $\rho _{j}^{\varepsilon }$ ($j=1,2$), $A_{i}^{\varepsilon }$ and $%
B_{i}^{\varepsilon }$ ($i=0,1$) defined by $\rho _{j}^{\varepsilon }(x)=\rho
_{j}(x,\frac{x}{\varepsilon })$, $A_{i}^{\varepsilon }(x,t)=A_{i}(x,t,\frac{x%
}{\varepsilon },\frac{t}{\varepsilon })$ and $B_{i}^{\varepsilon
}(x,t)=B_{i}(x,t,\frac{x}{\varepsilon },\frac{t}{\varepsilon })$ for $%
(x,t)\in Q=\Omega \times (0,T)$ ($T$ a given positive real number) are
constrained as follows:

\begin{itemize}
\item[(\textbf{A1})] $A_{i},B_{i}\in \mathcal{C}(\overline{Q};L^{\infty }(%
\mathbb{R}_{y}^{N}\times \mathbb{R}_{\tau }^{+})^{N\times N})$ are $N\times N
$ symmetric matrices satisfying the following assumption: 
\begin{equation*}
A_{0}\xi \cdot \xi \geq \alpha \left\vert \xi \right\vert ^{2},\ B_{0}\xi
\cdot \xi \geq \alpha \left\vert \xi \right\vert ^{2}\text{ for all }\xi \in 
\mathbb{R}^{N}\text{ and a.e. in }\overline{Q}\times \mathbb{R}_{y,\tau
}^{N+1}
\end{equation*}%
where $\alpha >0$ is a given constant not depending on $x,t,y,\tau $ and $%
\xi $, $\mathbb{R}_{\zeta }^{m}$ (integer $m\geq 1$) being denoting the
numerical space $\mathbb{R}^{m}$ with variables $\zeta \in \mathbb{R}^{m}$
and $\mathbb{R}_{\tau }^{+}=\mathbb{R}_{\tau }\cap \lbrack 0,\infty )$;

\item[(\textbf{A2})] $\rho _{j}\in \mathcal{C}(\overline{\Omega };L^{\infty
}(\mathbb{R}_{y}^{N}))$ and there exists $\Lambda >0$ such that $\Lambda
^{-1}\leq \rho _{j}\leq \Lambda $ a.e. in $\overline{\Omega }\times \mathbb{R%
}_{y}^{N}$.

\item[(\textbf{A3})] \textbf{Periodicity}. The characteristic functions $%
\chi _{Y_{2}}$ and $\chi _{Z_{2}}$ of the sets $Y_{2}$ and $Z_{2}$ are $%
[0,1)^{N}$-periodic, and the functions $A_{i}(x,t,\cdot ,\cdot )$, $%
B_{i}(x,t,\cdot ,\cdot )$ and $\rho _{j}(x,\cdot )$ are periodic in the
following sense: 
\begin{equation}
A_{i}(x,t,\cdot ,\cdot ),\ B_{i}(x,t,\cdot ,\cdot )\in L_{\text{per}%
}^{\infty }(Y\times \mathcal{T})^{N\times N}\text{ for all }(x,t)\in 
\overline{Q}  \label{2.19}
\end{equation}%
\begin{equation}
\rho _{j}(x,\cdot )\in \mathcal{C}_{\text{per}}(Y)\text{ for all }x\in 
\overline{\Omega }  \label{2.20}
\end{equation}%
where $\mathcal{T}=[0,1)$ and the spaces $L_{\text{per}}^{p}$ ($1\leq p\leq
\infty $) and $\mathcal{C}_{\text{per}}$ are defined below.
\end{itemize}

\begin{remark}
\label{r2.2}\emph{If we denote by }$\chi _{i}^{\varepsilon }$\emph{\ (}$%
i=1,2 $\emph{) the characteristic function of the set }$\Omega
_{i}^{\varepsilon }$\emph{, then it is important to express }$\chi
_{1}^{\varepsilon }$\emph{\ in terms of the characteristic functions of the
sets }$Y_{2}$\emph{\ and }$Z_{2}$\emph{. Denoting by }$\chi
_{c}^{\varepsilon }$\emph{\ and }$\chi _{p}^{\varepsilon }$\emph{\ the
characteristic functions of the cracks and pores spaces in }$\mathbb{R}^{N}$%
\emph{\ respectively, we have }%
\begin{eqnarray*}
\chi _{c}^{\varepsilon }(x) &=&\chi _{G_{2}}\left( \frac{x}{\varepsilon }%
\right) \equiv \chi _{Y_{2}}\left( \frac{x}{\varepsilon }\right) \text{ 
\emph{(obtained by }}Y\text{\emph{-periodicity)}} \\
\chi _{p}^{\varepsilon }(x) &=&\left( 1-\chi _{G_{2}}\left( \frac{x}{%
\varepsilon }\right) \right) \chi _{H_{2}}\left( \frac{x}{\varepsilon ^{2}}%
\right) \\
&\equiv &\left( 1-\chi _{Y_{2}}\left( \frac{x}{\varepsilon }\right) \right)
\chi _{Z_{2}}\left( \frac{x}{\varepsilon ^{2}}\right) \text{\emph{\
(obtained by }}Z\text{\emph{-periodicity),}}
\end{eqnarray*}%
\emph{hence }%
\begin{equation*}
\chi _{2}^{\varepsilon }(x)=\chi _{c}^{\varepsilon }(x)+\chi
_{p}^{\varepsilon }(x)=\chi _{Y_{2}}\left( \frac{x}{\varepsilon }\right)
+\left( 1-\chi _{Y_{2}}\left( \frac{x}{\varepsilon }\right) \right) \chi
_{Z_{2}}\left( \frac{x}{\varepsilon ^{2}}\right)
\end{equation*}%
\emph{and }%
\begin{equation*}
\chi _{1}^{\varepsilon }(x)=1-\chi _{2}^{\varepsilon }(x)\text{\emph{, for }}%
x\in \Omega .\ \ \ \ \ \ \ \ \ \ \ \ \ \ \ \ \ \ \ \ \ \ \ \ \ \ 
\end{equation*}
\end{remark}

Finally we assume that the functions $f_{1},f_{2}\in L^{2}(Q)^{N}$ and $%
\boldsymbol{u}^{0},\boldsymbol{v}^{0}\in L^{2}(\Omega )^{N}$ with $\Div%
\boldsymbol{u}^{0}=\Div\boldsymbol{v}^{0}=0$ in the sense of the
distributions in $\Omega $.

Our main objective in this work is to find the limiting behavior when $%
\varepsilon \rightarrow 0$, of the sequence of solutions to the system (\ref%
{2.1})-(\ref{2.8}). In this respect, we prove the following result.

\begin{itemize}
\item Assuming that (\textbf{A}1)-(\textbf{A}3) hold true, let (for any $%
\varepsilon >0$) let $\boldsymbol{u}_{\varepsilon }$ (resp. $\boldsymbol{v}%
_{\varepsilon }$) be the velocity field of the fluid in $\Omega
_{1}^{\varepsilon }$ (resp. in $\Omega _{2}^{\varepsilon }$). Let $\pi
_{\varepsilon }=\chi _{1}^{\varepsilon }p_{\varepsilon }+\chi
_{2}^{\varepsilon }q_{\varepsilon }$ (where $\chi _{j}^{\varepsilon }$ is
the characteristic function of the set $\Omega _{j}^{\varepsilon }$) be the
global pressure and set 
\begin{equation*}
\mathbf{f}(x,t)=\iint_{Y\times Z}(\chi _{1}(y,z)\rho
_{1}(x,y)f_{1}(x,t)+\chi _{2}(y,z)\rho _{2}(x,y)f_{2}(x,t))dydz,\ \text{a.e. 
}(x,t)\in Q.
\end{equation*}%
There exist $\boldsymbol{u}\in L^{\infty }(0,T;L^{2}(\Omega )^{N})$ -- the
velocity of the fluid in the skeleton, $\boldsymbol{v}\in L^{\infty
}(0,T;L^{2}(\Omega )^{N})$ -- the velocity of the fluid in the pores and
cracks system, and $p\in L^{2}(0,T;L^{2}(\Omega )/\mathbb{R})$ such that, as 
$\varepsilon \rightarrow 0$, $\chi _{1}^{\varepsilon }\boldsymbol{u}%
_{\varepsilon }\rightarrow \boldsymbol{u}$ in $L^{2}(Q)^{N}$-weak, $\chi
_{2}^{\varepsilon }\boldsymbol{v}_{\varepsilon }\rightarrow \boldsymbol{v}$
in $L^{2}(Q)^{N}$-weak and $\pi _{\varepsilon }\rightarrow p$ in $L^{2}(Q)$%
-weak. Moreover $\boldsymbol{u}=(1-m_{c})(1-m_{p})\boldsymbol{u}_{0}$ and $%
\boldsymbol{v}=\boldsymbol{v}_{c}+\boldsymbol{v}_{p}$ where $\boldsymbol{v}%
_{c}=m_{c}\boldsymbol{u}_{0}$ is the velocity of the fluid in the crack
space and $\boldsymbol{v}_{p}=(1-m_{c})m_{p}\boldsymbol{u}_{0}$ is the
velocity of the fluid in the pore space, and $m_{p}$ (resp. $m_{c}$) is the
porosity of the pore (resp. crack) space and $(\boldsymbol{u}_{0},p)$ is the
unique solution to problem 
\begin{equation*}
\left\{ 
\begin{array}{l}
\rho \frac{\partial \boldsymbol{u}_{0}}{\partial t}-\Div\left( \mathcal{A}%
_{0}\nabla \boldsymbol{u}_{0}+\int_{0}^{t}\mathcal{A}_{1}(x,t-\tau )\nabla 
\boldsymbol{u}_{0}(x,\tau )d\tau \right) +\nabla p=\mathbf{f}\text{ in }Q \\ 
\ \ \ \ \ \ \ \ \ \ \ \ \ \ \ \ \ \ \ \ \ \ \ \ \ \ \ \ \ \ \ \ \ \ \ \ \ \
\ \ \ \ \ \Div\boldsymbol{u}_{0}=0\text{ in }Q \\ 
\ \ \ \ \ \ \ \ \ \ \ \ \ \ \ \ \ \ \ \ \ \ \ \ \ \ \ \ \ \ \ \ \ \ \ 
\boldsymbol{u}_{0}=0\text{ on }\partial \Omega \times (0,T) \\ 
\boldsymbol{u}_{0}(x,0)=(1-m_{c})(1-m_{p})\boldsymbol{u}%
^{0}(x)+(m_{c}+m_{p}(1-m_{c}))\boldsymbol{v}^{0}(x),\ x\in \Omega .%
\end{array}%
\right.
\end{equation*}
\end{itemize}

\bigskip The rest of the paper is organized as follows. In Section 2, we
prove an existence result and a compactness result. In Section 3, we give
necessary material about multiscale convergence together with its connection
to convolution. Finally, Section 4 deals with the derivation of the limiting
model in which we prove the main result of the paper.

We end this section with some notations. All functions are assumed real
values and all function spaces are considered over $\mathbb{R}$. Let $%
Y=[0,1)^{N}$ and let $F(\mathbb{R}^{N})$ be a given function space. In the
case when $F$\ is either $L^{p}$\ or $W^{1,p}$\ ($1\leq p\leq \infty $), we
denote by $F_{\text{per}}(Y)$ the space of functions in $F_{\text{loc}}(%
\mathbb{R}^{N})$ that are $Y$-periodic.\textbf{\ }For $F=\mathcal{C}$, we
denote by\textbf{\ }$\mathcal{C}_{\text{per}}(Y)$\ the space of continuous
functions over $\mathbb{R}^{N}$\textbf{\ }which are $Y$-periodic. We denote
by $F_{\#}(Y)$ the subspace of $F_{\text{per}}(Y)$ consisting of functions $%
u $ having mean value zero: $\int_{Y}u(y)dy=0$. To wit, $W_{\#}^{1,p}(Y)$
stands for the space of those functions $u\in W_{\text{loc}}^{1,p}(\mathbb{R}%
^{N})$ which are $Y$-periodic and satisfy $\int_{Y}u(y)dy=0$. As special
case, $\mathcal{C}_{\text{per}}^{\infty }(Y)=\mathcal{C}_{\text{per}}(Y)\cap 
\mathcal{C}^{\infty }(\mathbb{R}^{N})$. Accordingly, we set $Z=[0,1)^{N}$, $%
\mathcal{T}=[0,1)$, and we define the corresponding spaces. Let $A$ be a $%
m\times m$ matrix whose entries are functions of unknowns $w$, and let $%
\mathbf{u}$ be either a $m\times 1$ vector function or a $m\times m$ matrix.
We will denote by $A(w)$ the value of $A$ at $w$, while either $A\mathbf{u}$
or $A[\mathbf{u}]$ will stand for the product of the $m\times m$ matrix $A$
by $\mathbf{u}$. If $\xi =(\xi _{ij})_{1\leq i,j\leq N}$ and $\eta =(\eta
_{ij})_{1\leq i,j\leq N}\in \mathbb{R}^{N^{2}}$, we define the product $\xi
\cdot \eta $ by $\xi \cdot \eta =\sum_{i,j=1}^{N}\xi _{ij}\eta _{ij}$.

\section{Existence result and uniform estimates}

Our first aim is to give an existence result. The main classical spaces
involved in the mathematical study of incompressible fluid flows are spaces
connected to kinetic energy, the boundary conditions and the conservation of
mass. These spaces are here defined as follows:

\begin{equation*}
V_{\varepsilon }=\{(v_{1},v_{2})\in V_{\varepsilon }^{1}\times
V_{\varepsilon }^{2}:\gamma _{1}^{\varepsilon }v_{1}=\gamma
_{2}^{\varepsilon }v_{2}\text{ on }\Gamma _{12}^{\varepsilon }\}
\end{equation*}%
and 
\begin{equation*}
H_{\varepsilon }=\{(v_{1},v_{2})\in L^{2}(\Omega _{1}^{\varepsilon
})^{N}\times L^{2}(\Omega _{2}^{\varepsilon })^{N}:\Div v_{j}=0\text{ in }%
\Omega _{j}^{\varepsilon }\text{ and }v_{j}\cdot \nu _{j}=0\text{ on }%
\partial \Omega _{j}^{\varepsilon }\}
\end{equation*}%
where 
\begin{equation*}
V_{\varepsilon }^{j}=\{v\in H^{1}(\Omega _{j}^{\varepsilon })^{N}:\Div %
v_{j}=0\text{ in }\Omega _{j}^{\varepsilon }\text{ and }\gamma
_{j}^{\varepsilon }v_{j}=0\text{ on }\partial \Omega _{j}^{\varepsilon }\cap
\partial \Omega \}\ (j=1,2),
\end{equation*}%
$\gamma _{j}^{\varepsilon }$ being denoting the zero order trace on the
boundary $\partial \Omega _{j}^{\varepsilon }$ of $\Omega _{j}^{\varepsilon
} $. The space $H_{\varepsilon }$ is a Hilbert space with Hilbertian norm 
\begin{equation*}
\left\Vert (v_{1},v_{2})\right\Vert _{H_{\varepsilon }}=\left(
\sum_{j=1}^{2}\left\Vert v_{j}\right\Vert _{L^{2}(\Omega _{j}^{\varepsilon
})^{N}}^{2}\right) ^{\frac{1}{2}}.
\end{equation*}%
For $(v_{1},v_{2})\in V_{\varepsilon }$ we set 
\begin{equation*}
\left\Vert (v_{1},v_{2})\right\Vert _{V_{\varepsilon }}=\left(
\sum_{j=1}^{2}\left\Vert \nabla v_{j}\right\Vert _{L^{2}(\Omega
_{j}^{\varepsilon })^{N^{2}}}^{2}\right) ^{\frac{1}{2}}.
\end{equation*}%
The following holds true.

\begin{lemma}
\label{l0}Equipped with $\left\Vert \cdot \right\Vert _{V_{\varepsilon }}$, $%
V_{\varepsilon }$ is a Hilbert space.
\end{lemma}

\begin{proof}
It is sufficient to verify that $\left\Vert \cdot \right\Vert
_{V_{\varepsilon }}$ is a norm on $V_{\varepsilon }$. To that end, let $%
(v_{1},v_{2})\in V_{\varepsilon }$, and set $u=\chi _{1}^{\varepsilon
}v_{1}+\chi _{2}^{\varepsilon }v_{2}$. Since $\gamma _{1}^{\varepsilon
}v_{1}=\gamma _{2}^{\varepsilon }v_{2}$ on $\Gamma _{12}^{\varepsilon }$ and 
$\gamma _{j}^{\varepsilon }v_{j}=0$ on $\partial \Omega _{j}^{\varepsilon
}\cap \partial \Omega $ ($j=1,2$), it holds that $\nabla u=\chi
_{1}^{\varepsilon }\nabla v_{1}+\chi _{2}^{\varepsilon }\nabla v_{2}$ and $%
u=0$ on $\partial \Omega $. Thus $u\in H_{0}^{1}(\Omega )^{N}$. Using the
general Poincar\'{e} inequality for punctured domains \cite[Theorem 4]{Lieb}%
, there exists a positive constant $C$ depending only on $\Omega $ such that 
\begin{eqnarray*}
\left\Vert u-\frac{1}{\left\vert \Omega \right\vert }\int_{\Omega
}udx\right\Vert _{L^{2}(\Omega )} &\leq &C\left[ \left\Vert \nabla
u\right\Vert _{L^{2}(\Omega _{1}^{\varepsilon })}+\left\Vert \nabla
u\right\Vert _{L^{2}(\Omega _{2}^{\varepsilon })}\right] \\
&=&C\left[ \left\Vert \nabla v_{1}\right\Vert _{L^{2}(\Omega
_{1}^{\varepsilon })}+\left\Vert \nabla v_{2}\right\Vert _{L^{2}(\Omega
_{2}^{\varepsilon })}\right] .
\end{eqnarray*}%
Now assume $\left\Vert (v_{1},v_{2})\right\Vert _{V_{\varepsilon }}=0$; then 
$\left\Vert \nabla v_{j}\right\Vert _{L^{2}(\Omega _{j}^{\varepsilon })}=0$ (%
$j=1,2$), hence 
\begin{equation*}
u-\frac{1}{\left\vert \Omega \right\vert }\int_{\Omega }udx=0\text{ a.e. in }%
\Omega .
\end{equation*}%
We infer that $u=0$ a.e. in $\Omega $ since $u\in H_{0}^{1}(\Omega )^{N}$.
This entails $\chi _{j}^{\varepsilon }u=0$, i.e. $v_{j}=0$ a.e. in $\Omega
_{j}^{\varepsilon }$. This is sufficient to conclude that $\left\Vert \cdot
\right\Vert _{V_{\varepsilon }}$ is a norm on $V_{\varepsilon }$ since the
other properties are easily verified.
\end{proof}

Now, we assume in the sequel that $H_{\varepsilon }$ is rather equipped with
the inner product 
\begin{equation*}
\left( \left( \boldsymbol{u},\boldsymbol{v}\right) \right) =\int_{\Omega
}(\chi _{1}^{\varepsilon }\rho _{1}^{\varepsilon }u_{1}\cdot v_{1}+\chi
_{2}^{\varepsilon }\rho _{2}^{\varepsilon }u_{2}\cdot v_{2})dx\text{ for }%
\boldsymbol{u}=(u_{1},u_{2}),\boldsymbol{v}=(v_{1},v_{2})\in H_{\varepsilon
},
\end{equation*}%
which makes it a Hilbert space. This stems from the inequality $\Lambda
^{-1}\leq \rho _{j}^{\varepsilon }\leq \Lambda $ a.e. in $\Omega $. Keeping
this in mind, the following continuous embeddings $V_{\varepsilon
}\hookrightarrow H_{\varepsilon }\hookrightarrow V_{\varepsilon }^{\prime }$
hold true.

Now, set $U_{\varepsilon }=(\boldsymbol{u}_{\varepsilon },\boldsymbol{v}%
_{\varepsilon })$, $\mathcal{F}=(f_{1},f_{2})$ and $U^{0}=(\boldsymbol{u}%
^{0},\boldsymbol{v}^{0})$. If we choose $\boldsymbol{v}=(v_{1},v_{2})\in
V_{\varepsilon }$ and multiply Eqns (\ref{2.1}) and (\ref{2.3}) by $v_{1}$
and $v_{2}$ respectively, and next sum up the resulting equations, we get 
\begin{eqnarray}
&&\frac{d}{dt}\left( \left( U_{\varepsilon }(t),\boldsymbol{v}\right)
\right) +\int_{\Omega }\left( \chi _{1}^{\varepsilon }A_{0}^{\varepsilon
}(t)\nabla \boldsymbol{u}_{\varepsilon }(t)\cdot \nabla v_{1}+\chi
_{2}^{\varepsilon }B_{0}^{\varepsilon }(t)\nabla \boldsymbol{v}_{\varepsilon
}(t)\cdot \nabla v_{2}\right) dx  \label{2.9} \\
&&+\int_{0}^{t}\left( \int_{\Omega }\left( \chi _{1}^{\varepsilon
}A_{1}^{\varepsilon }(t-\tau )\nabla \boldsymbol{u}_{\varepsilon }(\tau
)\cdot \nabla v_{1}+\chi _{2}^{\varepsilon }B_{1}^{\varepsilon }(t-\tau
)\nabla \boldsymbol{v}_{\varepsilon }(\tau )\cdot \nabla v_{2}\right)
dx\right) d\tau  \notag \\
&=&\left( \left( \mathcal{F},\boldsymbol{v}\right) \right) .  \notag
\end{eqnarray}%
The linear mappings 
\begin{equation*}
V_{\varepsilon }\ni \boldsymbol{v}=(v_{1},v_{2})\rightarrow \int_{\Omega
}\left( \chi _{1}^{\varepsilon }A_{0}^{\varepsilon }(t)\nabla \boldsymbol{u}%
_{\varepsilon }(t)\cdot \nabla v_{1}+\chi _{2}^{\varepsilon
}B_{0}^{\varepsilon }(t)\nabla \boldsymbol{v}_{\varepsilon }(t)\cdot \nabla
v_{2}\right) dx
\end{equation*}%
and 
\begin{equation*}
V_{\varepsilon }\ni \boldsymbol{v}\rightarrow -\int_{\Omega }\left( \chi
_{1}^{\varepsilon }A_{1}^{\varepsilon }(t-\tau )\nabla \boldsymbol{u}%
_{\varepsilon }(\tau )\cdot \nabla v_{1}+\chi _{2}^{\varepsilon
}B_{1}^{\varepsilon }(t-\tau )\nabla \boldsymbol{v}_{\varepsilon }(\tau
)\cdot \nabla v_{2}\right) dx
\end{equation*}%
belong to $V_{\varepsilon }^{\prime }$ and hence define two bounded linear
operators $\mathcal{A}_{\varepsilon }(t)$ and $\mathcal{B}_{\varepsilon
}(t,\tau )$ (for a.e. $0\leq t\leq T$ and $0\leq \tau \leq t\leq T$) from $%
V_{\varepsilon }$ into $V_{\varepsilon }^{\prime }$ as follows: 
\begin{eqnarray*}
\left\langle \mathcal{A}_{\varepsilon }(t)U_{\varepsilon }(t),\boldsymbol{v}%
\right\rangle &=&\int_{\Omega }\left( \chi _{1}^{\varepsilon
}A_{0}^{\varepsilon }(t)\nabla \boldsymbol{u}_{\varepsilon }(t)\cdot \nabla
v_{1}+\chi _{2}^{\varepsilon }B_{0}^{\varepsilon }(t)\nabla \boldsymbol{v}%
_{\varepsilon }(t)\cdot \nabla v_{2}\right) dx, \\
\left\langle \mathcal{B}_{\varepsilon }(t,\tau )U_{\varepsilon }(\tau ),%
\boldsymbol{v}\right\rangle &=&-\int_{\Omega }\left( \chi _{1}^{\varepsilon
}A_{1}^{\varepsilon }(t-\tau )\nabla \boldsymbol{u}_{\varepsilon }(\tau
)\cdot \nabla v_{1}+\chi _{2}^{\varepsilon }B_{1}^{\varepsilon }(t-\tau
)\nabla \boldsymbol{v}_{\varepsilon }(\tau )\cdot \nabla v_{2}\right) dx \\
\text{for }v &\in &V_{\varepsilon }.
\end{eqnarray*}%
Therefore (\ref{2.9}) tantamount to 
\begin{equation}
\frac{dU_{\varepsilon }}{dt}(t)+\mathcal{A}_{\varepsilon }(t)U_{\varepsilon
}(t)=\int_{0}^{t}\mathcal{B}_{\varepsilon }(t,\tau )U_{\varepsilon }(\tau
)d\tau +\mathcal{F}\text{ in }V_{\varepsilon }^{\prime }  \label{2.10}
\end{equation}%
\begin{equation}
U_{\varepsilon }(0)=U^{0}\text{ in }H_{\varepsilon }.\ \ \ \ \ \ \ \ \ \ \ \
\ \ \ \ \ \ \ \ \ \ \ \ \ \ \ \ \ \ \ \ \ \   \label{2.11}
\end{equation}%
The following result holds.

\begin{proposition}
\label{p2.1}It holds that 
\begin{equation}
\left\langle \mathcal{A}_{\varepsilon }(t)\boldsymbol{v},\boldsymbol{v}%
\right\rangle \geq \alpha \left\Vert \boldsymbol{v}\right\Vert
_{V_{\varepsilon }}^{2}\text{ for all }\boldsymbol{v}\in V_{\varepsilon }.
\label{2.12}
\end{equation}
\end{proposition}

\begin{proof}
For $\boldsymbol{v}=(v_{1},v_{2})\in V_{\varepsilon }$ we have 
\begin{eqnarray*}
\left\langle \mathcal{A}_{\varepsilon }(t)\boldsymbol{v},\boldsymbol{v}%
\right\rangle &=&\int_{\Omega }\left( \chi _{1}^{\varepsilon
}A_{0}^{\varepsilon }(t)\nabla v_{1}\cdot \nabla v_{1}+\chi
_{2}^{\varepsilon }B_{0}^{\varepsilon }(t)\nabla v_{2}\cdot \nabla
v_{2}\right) dx \\
&\geq &\alpha \int_{\Omega }\left( \chi _{1}^{\varepsilon }\left\vert \nabla
v_{1}\right\vert ^{2}+\chi _{2}^{\varepsilon }\left\vert \nabla
v_{2}\right\vert ^{2}\right) dx=\alpha \left\Vert \boldsymbol{v}\right\Vert
_{V_{\varepsilon }}^{2}.
\end{eqnarray*}
\end{proof}

We can now state and prove the existence result.

\begin{theorem}
\label{t2.1}For any $\varepsilon >0$, there exist two pairs $(\boldsymbol{u}%
_{\varepsilon },\boldsymbol{v}_{\varepsilon })\in L^{2}(0,T;V_{\varepsilon
})\cap L^{\infty }(0,T;H_{\varepsilon })$ and $(p_{\varepsilon
},q_{\varepsilon })\in L^{2}(0,T;L^{2}(\Omega _{1}^{\varepsilon })\times
L^{2}(\Omega _{2}^{\varepsilon }))$ that solve \emph{(\ref{2.1})-(\ref{2.8})}%
. Moreover the vector-function $(\boldsymbol{u}_{\varepsilon },\boldsymbol{v}%
_{\varepsilon })$ is unique and belongs to $\mathcal{C}([0,T];H_{\varepsilon
})$, and $(p_{\varepsilon },q_{\varepsilon })$ is unique up to a constant in
the following sense: 
\begin{equation}
\int_{\Omega _{1}^{\varepsilon }}p_{\varepsilon }dx=0\text{ and }%
\int_{\Omega _{2}^{\varepsilon }}q_{\varepsilon }dx=0.  \label{2.13}
\end{equation}
\end{theorem}

\begin{proof}
We infer from Proposition \ref{p2.1} that the hypotheses of Theorem 3.2 in 
\cite{Orlik} are fulfilled. Therefore, appealing to the above cited result
we obtain the existence and uniqueness of $(\boldsymbol{u}_{\varepsilon },%
\boldsymbol{v}_{\varepsilon })\in L^{2}(0,T;V_{\varepsilon })\cap \mathcal{C}%
([0,T];H_{\varepsilon })$ satisfying (\ref{2.10})-(\ref{2.11}). The
existence of $p_{\varepsilon }$ and $q_{\varepsilon }$ satisfying (\ref{2.13}%
) follows by the use of Propositions 1.1 and 1.2 of \cite{Temam}.
\end{proof}

The next result provides us with uniform estimates.

\begin{lemma}
\label{l2.1}Under assumptions \emph{(\textbf{A1})-(\textbf{A2})} it holds
that 
\begin{equation}
\sup_{0\leq t\leq T}\left( \left\Vert \boldsymbol{u}_{\varepsilon
}(t)\right\Vert _{L^{2}(\Omega _{1}^{\varepsilon })}^{2}+\left\Vert 
\boldsymbol{v}_{\varepsilon }(t)\right\Vert _{L^{2}(\Omega _{2}^{\varepsilon
})}^{2}\right) \leq C,  \label{2.14}
\end{equation}%
\begin{equation}
\int_{0}^{T}\left( \left\Vert \nabla \boldsymbol{u}_{\varepsilon
}(t)\right\Vert _{L^{2}(\Omega _{1}^{\varepsilon })}^{2}+\left\Vert \nabla 
\boldsymbol{v}_{\varepsilon }(t)\right\Vert _{L^{2}(\Omega _{2}^{\varepsilon
})}^{2}\right) dt\leq C  \label{2.15}
\end{equation}%
\begin{equation}
\left\Vert p_{\varepsilon }\right\Vert _{L^{2}(\Omega _{1}^{\varepsilon
}\times (0,T))}\leq C\text{ and }\left\Vert q_{\varepsilon }\right\Vert
_{L^{2}(\Omega _{2}^{\varepsilon }\times (0,T))}\leq C  \label{2.16}
\end{equation}%
\begin{equation}
\left\Vert \rho _{1}^{\varepsilon }\frac{\partial \boldsymbol{u}%
_{\varepsilon }}{\partial t}\right\Vert _{L^{2}(0,T;(V_{\varepsilon
}^{1})^{\prime })}\leq C\text{ and }\left\Vert \rho _{2}^{\varepsilon }\frac{%
\partial \boldsymbol{v}_{\varepsilon }}{\partial t}\right\Vert
_{L^{2}(0,T;(V_{\varepsilon }^{2})^{\prime })}\leq C  \label{2.17}
\end{equation}%
where $C$ is a positive constant not depending on $\varepsilon $.
\end{lemma}

\begin{proof}
We multiply (\ref{2.1}) and (\ref{2.3}) respectively by $\boldsymbol{u}%
_{\varepsilon }$ and $\boldsymbol{v}_{\varepsilon }$. Then denoting by $\ast 
$ the convolution with respect to the time variable $t$, it holds that 
\begin{equation*}
\begin{array}{l}
\left( \rho _{1}^{\varepsilon }\frac{\partial \boldsymbol{u}_{\varepsilon }}{%
\partial t},\boldsymbol{u}_{\varepsilon }\right) +\left( \rho
_{2}^{\varepsilon }\frac{\partial \boldsymbol{v}_{\varepsilon }}{\partial t},%
\boldsymbol{v}_{\varepsilon }\right) +\left( A_{0}^{\varepsilon }\nabla 
\boldsymbol{u}_{\varepsilon }+A_{1}^{\varepsilon }\ast \nabla \boldsymbol{u}%
_{\varepsilon },\nabla \boldsymbol{u}_{\varepsilon }\right) \\ 
\\ 
\ \ \ \ \ +\left( B_{0}^{\varepsilon }\nabla \boldsymbol{v}_{\varepsilon
}+B_{1}^{\varepsilon }\ast \nabla \boldsymbol{v}_{\varepsilon },\nabla 
\boldsymbol{v}_{\varepsilon }\right) =(\rho _{1}^{\varepsilon }f_{1},%
\boldsymbol{u}_{\varepsilon })+(\rho _{2}^{\varepsilon }f_{2},\boldsymbol{v}%
_{\varepsilon }),%
\end{array}%
\end{equation*}%
or equivalently, 
\begin{equation*}
\begin{array}{l}
\frac{1}{2}\frac{d}{dt}\left\Vert (\rho _{1}^{\varepsilon })^{\frac{1}{2}}%
\boldsymbol{u}_{\varepsilon }(t)\right\Vert _{L^{2}(\Omega _{1}^{\varepsilon
})}^{2}+\frac{1}{2}\frac{d}{dt}\left\Vert (\rho _{2}^{\varepsilon })^{\frac{1%
}{2}}\boldsymbol{v}_{\varepsilon }(t)\right\Vert _{L^{2}(\Omega
_{2}^{\varepsilon })}^{2} \\ 
\\ 
\ \ +\left( A_{0}^{\varepsilon }\nabla \boldsymbol{u}_{\varepsilon
}+A_{1}^{\varepsilon }\ast \nabla \boldsymbol{u}_{\varepsilon },\nabla 
\boldsymbol{u}_{\varepsilon }\right) +\left( B_{0}^{\varepsilon }\nabla 
\boldsymbol{v}_{\varepsilon }+B_{1}^{\varepsilon }\ast \nabla \boldsymbol{v}%
_{\varepsilon },\nabla \boldsymbol{v}_{\varepsilon }\right) \\ 
\\ 
\ \ \ \ \ =(\rho _{1}^{\varepsilon }f_{1},\boldsymbol{u}_{\varepsilon
})+(\rho _{2}^{\varepsilon }f_{2},\boldsymbol{v}_{\varepsilon }).%
\end{array}%
\end{equation*}%
Integrating with respect to $t$, 
\begin{equation*}
\begin{array}{l}
\left\Vert (\rho _{1}^{\varepsilon })^{\frac{1}{2}}\boldsymbol{u}%
_{\varepsilon }(t)\right\Vert _{L^{2}(\Omega _{1}^{\varepsilon
})}^{2}+\left\Vert (\rho _{2}^{\varepsilon })^{\frac{1}{2}}\boldsymbol{v}%
_{\varepsilon }(t)\right\Vert _{L^{2}(\Omega _{2}^{\varepsilon })}^{2} \\ 
\\ 
\ \ +2\int_{0}^{t}\left( A_{0}^{\varepsilon }\nabla \boldsymbol{u}%
_{\varepsilon }+A_{1}^{\varepsilon }\ast \nabla \boldsymbol{u}_{\varepsilon
},\nabla \boldsymbol{u}_{\varepsilon }\right) d\tau +2\int_{0}^{t}\left(
B_{0}^{\varepsilon }\nabla \boldsymbol{v}_{\varepsilon }+B_{1}^{\varepsilon
}\ast \nabla \boldsymbol{v}_{\varepsilon },\nabla \boldsymbol{v}%
_{\varepsilon }\right) d\tau \\ 
\\ 
\ \ \ \ \ =2\int_{0}^{t}\left[ (\rho _{1}^{\varepsilon }f_{1},\boldsymbol{u}%
_{\varepsilon })+(\rho _{2}^{\varepsilon }f_{2},\boldsymbol{v}_{\varepsilon
})\right] d\tau +\left\Vert (\rho _{1}^{\varepsilon })^{\frac{1}{2}}%
\boldsymbol{u}^{0}\right\Vert _{L^{2}(\Omega _{1}^{\varepsilon
})}^{2}+\left\Vert (\rho _{2}^{\varepsilon })^{\frac{1}{2}}\boldsymbol{v}%
^{0}\right\Vert _{L^{2}(\Omega _{2}^{\varepsilon })}^{2}.%
\end{array}%
\end{equation*}%
But 
\begin{eqnarray*}
&&2\int_{0}^{t}\left[ (\rho _{1}^{\varepsilon }f_{1},\boldsymbol{u}%
_{\varepsilon })+(\rho _{2}^{\varepsilon }f_{2},\boldsymbol{v}_{\varepsilon
})\right] d\tau \\
&\leq &\int_{0}^{t}\left( \left\Vert (\rho _{1}^{\varepsilon })^{\frac{1}{2}%
}f_{1}(\tau )\right\Vert _{L^{2}(\Omega _{1}^{\varepsilon })}^{2}+\left\Vert
(\rho _{2}^{\varepsilon })^{\frac{1}{2}}f_{2}(\tau )\right\Vert
_{L^{2}(\Omega _{2}^{\varepsilon })}^{2}\right) d\tau \\
&&+\int_{0}^{t}\left( \left\Vert (\rho _{1}^{\varepsilon })^{\frac{1}{2}}%
\boldsymbol{u}_{\varepsilon }(\tau )\right\Vert _{L^{2}(\Omega
_{1}^{\varepsilon })}^{2}+\left\Vert (\rho _{2}^{\varepsilon })^{\frac{1}{2}}%
\boldsymbol{v}_{\varepsilon }(\tau )\right\Vert _{L^{2}(\Omega
_{2}^{\varepsilon })}^{2}\right) d\tau .
\end{eqnarray*}%
Making use of (\textbf{A1}) we get 
\begin{eqnarray*}
&&\left\Vert (\rho _{1}^{\varepsilon })^{\frac{1}{2}}\boldsymbol{u}%
_{\varepsilon }(t)\right\Vert _{L^{2}(\Omega _{1}^{\varepsilon
})}^{2}+\left\Vert (\rho _{2}^{\varepsilon })^{\frac{1}{2}}\boldsymbol{v}%
_{\varepsilon }(t)\right\Vert _{L^{2}(\Omega _{2}^{\varepsilon })}^{2} \\
&&+2\alpha \int_{0}^{t}\left( \left\Vert \nabla \boldsymbol{u}_{\varepsilon
}(\tau )\right\Vert _{L^{2}(\Omega _{1}^{\varepsilon })}^{2}+\left\Vert
\nabla \boldsymbol{v}_{\varepsilon }(\tau )\right\Vert _{L^{2}(\Omega
_{2}^{\varepsilon })}^{2}\right) d\tau \\
&\leq &-2\int_{0}^{t}\left[ \left( A_{1}^{\varepsilon }\ast \nabla 
\boldsymbol{u}_{\varepsilon },\nabla \boldsymbol{u}_{\varepsilon }\right)
+\left( B_{1}^{\varepsilon }\ast \nabla \boldsymbol{v}_{\varepsilon },\nabla 
\boldsymbol{v}_{\varepsilon }\right) \right] d\tau \\
&&\ \ +\int_{0}^{t}\left( \left\Vert (\rho _{1}^{\varepsilon })^{\frac{1}{2}}%
\boldsymbol{u}_{\varepsilon }(\tau )\right\Vert _{L^{2}(\Omega
_{1}^{\varepsilon })}^{2}+\left\Vert (\rho _{2}^{\varepsilon })^{\frac{1}{2}}%
\boldsymbol{v}_{\varepsilon }(\tau )\right\Vert _{L^{2}(\Omega
_{2}^{\varepsilon })}^{2}\right) d\tau \\
&&\ \ \ \ \ +\int_{0}^{T}\left( \left\Vert (\rho _{1}^{\varepsilon })^{\frac{%
1}{2}}f_{1}(t)\right\Vert _{L^{2}(\Omega _{1}^{\varepsilon
})}^{2}+\left\Vert (\rho _{2}^{\varepsilon })^{\frac{1}{2}%
}f_{2}(t)\right\Vert _{L^{2}(\Omega _{2}^{\varepsilon })}^{2}\right) dt.
\end{eqnarray*}%
Now, using Young's inequality, 
\begin{eqnarray*}
&&2\int_{0}^{t}\left( A_{1}^{\varepsilon }\ast \nabla \boldsymbol{u}%
_{\varepsilon },\nabla \boldsymbol{u}_{\varepsilon }\right) d\tau \\
&=&2\int_{0}^{t}\left( \int_{0}^{\tau }\left( \int_{\Omega _{1}^{\varepsilon
}}A_{1}^{\varepsilon }(\tau -s)\nabla \boldsymbol{u}_{\varepsilon }(s)\cdot
\nabla \boldsymbol{u}_{\varepsilon }(\tau )dx)\right) ds\right) d\tau \\
&\leq &2\int_{0}^{t}\left( C\int_{0}^{\tau }\left\Vert A_{1}^{\varepsilon
}(\tau -s)\nabla \boldsymbol{u}_{\varepsilon }(s)\right\Vert _{L^{2}(\Omega
_{1}^{\varepsilon })}^{2}ds+\int_{0}^{\tau }\frac{\alpha }{2\tau }\left\Vert
\nabla \boldsymbol{u}_{\varepsilon }(\tau )\right\Vert _{L^{2}(\Omega
_{1}^{\varepsilon })}^{2}ds\right) d\tau \\
&\leq &\alpha \int_{0}^{t}\left\Vert \nabla \boldsymbol{u}_{\varepsilon
}(\tau )\right\Vert _{L^{2}(\Omega _{1}^{\varepsilon })}^{2}d\tau
+C\int_{0}^{t}\left( \int_{0}^{\tau }\left\Vert \nabla \boldsymbol{u}%
_{\varepsilon }(s)\right\Vert _{L^{2}(\Omega _{1}^{\varepsilon
})}^{2}ds\right) d\tau ,
\end{eqnarray*}%
hence 
\begin{eqnarray*}
&&2\int_{0}^{t}\left[ \left( A_{1}^{\varepsilon }\ast \nabla \boldsymbol{u}%
_{\varepsilon },\nabla \boldsymbol{u}_{\varepsilon }\right) +\left(
B_{1}^{\varepsilon }\ast \nabla \boldsymbol{v}_{\varepsilon },\nabla 
\boldsymbol{v}_{\varepsilon }\right) \right] d\tau \\
&\leq &\alpha \int_{0}^{t}\left( \left\Vert \nabla \boldsymbol{u}%
_{\varepsilon }(\tau )\right\Vert _{L^{2}(\Omega _{1}^{\varepsilon
})}^{2}+\left\Vert \nabla \boldsymbol{v}_{\varepsilon }(\tau )\right\Vert
_{L^{2}(\Omega _{2}^{\varepsilon })}^{2}\right) d\tau \\
&&+C\int_{0}^{t}\left( \int_{0}^{\tau }\left( \left\Vert \nabla \boldsymbol{u%
}_{\varepsilon }(s)\right\Vert _{L^{2}(\Omega _{1}^{\varepsilon
})}^{2}+\left\Vert \nabla \boldsymbol{v}_{\varepsilon }(s)\right\Vert
_{L^{2}(\Omega _{2}^{\varepsilon })}^{2}\right) ds\right) d\tau .
\end{eqnarray*}%
Therefore 
\begin{eqnarray*}
&&\left\Vert (\rho _{1}^{\varepsilon })^{\frac{1}{2}}\boldsymbol{u}%
_{\varepsilon }(t)\right\Vert _{L^{2}(\Omega _{1}^{\varepsilon
})}^{2}+\left\Vert (\rho _{2}^{\varepsilon })^{\frac{1}{2}}\boldsymbol{v}%
_{\varepsilon }(t)\right\Vert _{L^{2}(\Omega _{2}^{\varepsilon })}^{2} \\
&&+\alpha \int_{0}^{t}\left( \left\Vert \nabla \boldsymbol{u}_{\varepsilon
}(\tau )\right\Vert _{L^{2}(\Omega _{1}^{\varepsilon })}^{2}+\left\Vert
\nabla \boldsymbol{v}_{\varepsilon }(\tau )\right\Vert _{L^{2}(\Omega
_{2}^{\varepsilon })}^{2}\right) d\tau \\
&\leq &C+\int_{0}^{t}\left( \left\Vert (\rho _{1}^{\varepsilon })^{\frac{1}{2%
}}\boldsymbol{u}_{\varepsilon }(\tau )\right\Vert _{L^{2}(\Omega
_{1}^{\varepsilon })}^{2}+\left\Vert (\rho _{2}^{\varepsilon })^{\frac{1}{2}}%
\boldsymbol{v}_{\varepsilon }(\tau )\right\Vert _{L^{2}(\Omega
_{2}^{\varepsilon })}^{2}\right) d\tau \\
&&+C\int_{0}^{t}\left( \int_{0}^{\tau }\left( \left\Vert \nabla \boldsymbol{u%
}_{\varepsilon }(s)\right\Vert _{L^{2}(\Omega _{1}^{\varepsilon
})}^{2}+\left\Vert \nabla \boldsymbol{v}_{\varepsilon }(s)\right\Vert
_{L^{2}(\Omega _{2}^{\varepsilon })}^{2}\right) ds\right) d\tau .
\end{eqnarray*}%
It follows from Gronwall's inequality that 
\begin{equation*}
\begin{array}{l}
\left\Vert (\rho _{1}^{\varepsilon })^{\frac{1}{2}}\boldsymbol{u}%
_{\varepsilon }(t)\right\Vert _{L^{2}(\Omega _{1}^{\varepsilon
})}^{2}+\left\Vert (\rho _{2}^{\varepsilon })^{\frac{1}{2}}\boldsymbol{v}%
_{\varepsilon }(t)\right\Vert _{L^{2}(\Omega _{2}^{\varepsilon })}^{2} \\ 
\\ 
\ \ \ +\alpha \int_{0}^{t}\left( \left\Vert \nabla \boldsymbol{u}%
_{\varepsilon }(\tau )\right\Vert _{L^{2}(\Omega _{1}^{\varepsilon
})}^{2}+\left\Vert \nabla \boldsymbol{v}_{\varepsilon }(\tau )\right\Vert
_{L^{2}(\Omega _{2}^{\varepsilon })}^{2}\right) d\tau \leq C%
\end{array}%
\end{equation*}%
for all $0\leq t\leq T$ and all $\varepsilon >0$, where $C$ is independent
of $\varepsilon $ and $t$. We therefore deduce (\ref{2.14}) and (\ref{2.15}%
). (\ref{2.17}) follows immediately, and (\ref{2.16}) is obtained by
repeating the same arguments used in \cite[Section 2.2]{JMS}
\end{proof}

The next result deals with the compactness of the global velocity defined
below by 
\begin{equation}
\boldsymbol{u}^{\varepsilon }=\chi _{1}^{\varepsilon }\boldsymbol{u}%
_{\varepsilon }+\chi _{2}^{\varepsilon }\boldsymbol{v}_{\varepsilon }.
\label{2.21}
\end{equation}%
$\ \ $ We also set $\rho ^{\varepsilon }=\chi _{1}^{\varepsilon }\rho
_{1}^{\varepsilon }+\chi _{2}^{\varepsilon }\rho _{2}^{\varepsilon }$,
recalling that $\chi _{j}^{\varepsilon }$ ($j=1,2$) denotes the
characteristic function of the open set $\Omega _{j}^{\varepsilon }$.

\begin{proposition}
\label{p2.2}Assume that the sequence $(\rho ^{\varepsilon })_{\varepsilon
>0} $ weakly $\ast $-converges in $L^{\infty }(\Omega )$ to some real
function $\rho $ as $\varepsilon \rightarrow 0$, with $\rho (x)\neq 0$ for
a.e. $x\in \Omega $. Then the sequence $(\boldsymbol{u}^{\varepsilon
})_{\varepsilon >0} $ is relatively compact in the space $L^{2}(Q)^{N}$.
\end{proposition}

\begin{proof}
The proof is copied on that of \cite[Proposition 2.1]{JMS2}.
\end{proof}

\section{Multiscale convergence and related convolution results}

The letter $E$ denotes throughout any ordinary sequence $(\varepsilon
_{n})_{n\in \mathbb{N}}$ with $0<\varepsilon _{n}\leq 1$ such that $%
\varepsilon _{n}\rightarrow 0$ as $n\rightarrow \infty $. $\varepsilon $
will denote a generic element of $E$, and "$\varepsilon _{n}\rightarrow 0$
as $n\rightarrow \infty $" will henceforth be merely denoted by "$%
\varepsilon \rightarrow 0$". We assume throughout this section that $\Omega $
is an open subset of $\mathbb{R}^{N}$. We also set $Y=Z=[0,1)^{N}$.

\begin{definition}
\label{d3.1}\emph{Let }$1\leq p<\infty $\emph{. (1) A sequence }$%
(u_{\varepsilon })_{\varepsilon >0}\subset L^{p}(\Omega )$\emph{\ is said to}
weakly multiscale converge in $L^{p}(\Omega )$ to some $u_{0}\in
L^{p}(\Omega \times Y\times Z)$\emph{\ if as }$\varepsilon \rightarrow 0$, 
\begin{equation*}
\int_{\Omega }u_{\varepsilon }(x)\psi \left( x,\frac{x}{\varepsilon },\frac{x%
}{\varepsilon ^{2}}\right) dx\rightarrow \iiint_{\Omega \times Y\times
Z}u_{0}(x,y,z)\psi (x,y,z)dxdydz
\end{equation*}%
\emph{for every }$\psi \in L^{p^{\prime }}(\Omega ;\mathcal{C}_{\text{\emph{%
per}}}(Y\times Z))$\emph{\ (}$1/p^{\prime }=1-1/p$\emph{). We express this
by writing }$u_{\varepsilon }\rightarrow u_{0}$\emph{\ reit. in }$%
L^{p}(\Omega )$\emph{-weak.}

\noindent \emph{(2) The sequence }$(u_{\varepsilon })_{\varepsilon
>0}\subset L^{p}(\Omega )$\emph{\ is said to} strongly multiscale converge
in $L^{p}(\Omega )$ to some $u_{0}\in L^{p}(\Omega \times Y\times Z)$\emph{\
if it is weakly multiscale convergent towards }$u_{0}$\emph{\ and further
satisfies the following condition: }%
\begin{equation*}
\left\Vert u_{\varepsilon }\right\Vert _{L^{p}(\Omega )}\rightarrow
\left\Vert u_{0}\right\Vert _{L^{p}(\Omega \times Y\times Z)}\text{\emph{\
as }}\varepsilon \rightarrow \emph{0}\text{\emph{.}}
\end{equation*}%
\emph{We denote this by writing }$u_{\varepsilon }\rightarrow u_{0}$\emph{\
reit. in }$L^{p}(\Omega )$\emph{-strong.}
\end{definition}

The above definition has been introduced in \cite{AB} for the case $p=2$ and
later generalized to any $1<p<\infty $ in \cite{LLW} (see also \cite{LLW1,
LNW}).

The following two results are worth recalling; see e.g. \cite[Theorems 3.1
and 3.5]{CMP} (see also \cite{AB} for the case $p=2$).

\begin{theorem}
\label{t3.1}Any bounded sequence $(u_{\varepsilon })_{\varepsilon \in
E}\subset L^{p}(\Omega )$ \emph{(}$1<p<\infty $\emph{)} possesses a
subsequence which is weakly multiscale convergent in $L^{p}(\Omega )$.
\end{theorem}

\begin{theorem}
\label{t3.2}Let $(u_{\varepsilon })_{\varepsilon \in E}$ be a bounded in $%
W^{1,p}(\Omega )$ \emph{(}$1<p<\infty $\emph{)}. Then there exist a
subsequence $E^{\prime }$ of $E$ and a triple of functions $u_{0}\in
W^{1,p}(\Omega )$, $u_{1}\in L^{p}(\Omega ;W_{\#}^{1,p}(Y))$ and $u_{2}\in
L^{p}(\Omega \times Y;W_{\#}^{1,p}(Z))$ such that, as $E^{\prime }\ni
\varepsilon \rightarrow 0$, 
\begin{equation*}
u_{\varepsilon }\rightarrow u_{0}\text{ in }W^{1,p}(\Omega )\text{-weak}
\end{equation*}%
\begin{equation*}
\frac{\partial u_{\varepsilon }}{\partial x_{i}}\rightarrow \frac{\partial
u_{0}}{\partial x_{i}}+\frac{\partial u_{1}}{\partial y_{i}}+\frac{\partial
u_{2}}{\partial z_{i}}\text{ reit. in }L^{p}(\Omega )\text{-weak, }1\leq
i\leq N.
\end{equation*}
\end{theorem}

The proof of the next result is copied on that of \cite[Theorem 6]{DPDE}.

\begin{theorem}
\label{t3.3}Let $1<p,q<\infty $ and $r\geq 1$ be such that $\frac{1}{r}=%
\frac{1}{p}+\frac{1}{q}\leq 1$. Suppose that $u_{\varepsilon }\rightarrow
u_{0}$ reit. in $L^{p}(\Omega )$-weak and $v_{\varepsilon }\rightarrow v_{0}$
reit. in $L^{q}(\Omega )$-strong, where $u_{0}\in L^{p}(\Omega \times
Y\times Z)$ and $v_{0}\in L^{q}(\Omega \times Y\times Z)$. Then $%
u_{\varepsilon }v_{\varepsilon }\rightarrow u_{0}v_{0}$ reit. in $%
L^{r}(\Omega )$-weak.
\end{theorem}

As an immediate consequence of the preceding result, the following holds
true.

\begin{corollary}
\label{c3.1}Let $(u_{\varepsilon })_{\varepsilon \in E}\subset L^{p}(\Omega
) $ and $(v_{\varepsilon })_{\varepsilon \in E}\subset L^{p^{\prime
}}(\Omega )\cap L^{\infty }(\Omega )$ ($1<p<\infty $ and $p^{\prime
}=p/(p-1) $) be two sequences such that: \emph{(i)} $u_{\varepsilon
}\rightarrow u_{0}$ reit. in $L^{p}(\Omega )$-weak; \emph{(ii)} $%
v_{\varepsilon }\rightarrow v_{0}$ reit. in $L^{p^{\prime }}(\Omega )$%
-strong; \emph{(iii)} $(v_{\varepsilon })_{\varepsilon \in E}$ is bounded in 
$L^{\infty }(\Omega )$. Then $u_{\varepsilon }v_{\varepsilon }\rightarrow
u_{0}v_{0}$ reit. in $L^{p}(\Omega )$-weak.
\end{corollary}

The following result establishes a relationship between the limit of a
multiscale convergent sequence and the limit of its translates.

\begin{theorem}
\label{t3.4}Let $(u_{\varepsilon })_{\varepsilon \in E}\subset L^{p}(\Omega
) $ be a sequence such that $u_{\varepsilon }\rightarrow u_{0}$ reit. in $%
L^{p}(\Omega )$-weak, where $u_{0}\in L^{p}(\Omega \times Y\times Z)$. Let $%
\xi \in \mathbb{R}^{N}$ and define $v_{\varepsilon }(x)=u_{\varepsilon
}(x+\xi )$ for $x\in \Omega -\xi $. Then there exist a subsequence $%
E^{\prime }$ of $E$ and a couple $(r,s)\in Y\times Z$ such that $%
v_{\varepsilon }\rightarrow v_{0}$ reit. in $L^{p}(\Omega )$-weak when $%
E^{\prime }\ni \varepsilon \rightarrow 0$, where $v_{0}\in L^{p}((\Omega
-\xi )\times Y\times Z)$ is defined by $v_{0}(x,y,z)=u_{0}(x+\xi ,y+r,z+s)$
for $(x,y,z)\in (\Omega -\xi )\times Y\times Z$.
\end{theorem}

\begin{proof}
Let $\mathcal{R}\left( \xi /\varepsilon ^{i}\right) =\xi /\varepsilon ^{i}-%
\left[ \xi /\varepsilon ^{i}\right] $ ($i=1,2$) where $\left[ \xi
/\varepsilon ^{i}\right] $ denotes the integer part of $\frac{\xi }{%
\varepsilon ^{i}}$. Then $\left( \mathcal{R}\left( \xi /\varepsilon
^{i}\right) \right) _{\varepsilon \in E}$ is a bounded sequence in $%
[0,1)^{N} $, hence there exist a subsequence $E^{\prime }$ of $E$ and $%
(r,s)\in Y\times Z$ ($Y=Z=[0,1)^{N}$) such that the following convergence
results hold in the usual topology of $\mathbb{R}$: 
\begin{equation}
\mathcal{R}\left( \frac{\xi }{\varepsilon }\right) \rightarrow r\text{ and }%
\mathcal{R}\left( \frac{\xi }{\varepsilon ^{2}}\right) \rightarrow s\text{
as }E^{\prime }\ni \varepsilon \rightarrow 0.  \label{3.1}
\end{equation}%
This being so, let $\varphi \in \mathcal{C}_{0}^{\infty }(\Omega -\xi )$ and 
$\psi \in \mathcal{C}_{\text{per}}(Y\times Z)$; then 
\begin{eqnarray*}
&&\int_{\Omega -\xi }u_{\varepsilon }(x+\xi )\varphi (x)\psi \left( \frac{x}{%
\varepsilon },\frac{x}{\varepsilon ^{2}}\right) dx \\
&=&\int_{\Omega }u_{\varepsilon }(x)\varphi (x-\xi )\psi \left( \frac{x-\xi 
}{\varepsilon },\frac{x-\xi }{\varepsilon ^{2}}\right) dx \\
&=&\int_{\Omega }u_{\varepsilon }(x)\varphi (x-\xi )\left[ \psi \left( \frac{%
x}{\varepsilon }-\frac{\xi }{\varepsilon },\frac{x}{\varepsilon ^{2}}-\frac{%
\xi }{\varepsilon ^{2}}\right) -\psi \left( \frac{x}{\varepsilon }-r,\frac{x%
}{\varepsilon ^{2}}-s\right) \right] dx \\
&&+\int_{\Omega }u_{\varepsilon }(x)\varphi (x-\xi )\psi \left( \frac{x}{%
\varepsilon }-r,\frac{x}{\varepsilon ^{2}}-s\right) dx \\
&=&(I)+(II).
\end{eqnarray*}%
Set $\phi (y,z)=\psi (y-r,z-s)$, $(y,z)\in Y\times Z$. Then $\phi \in 
\mathcal{C}_{\text{per}}(Y\times Z)$, and thus, as $E^{\prime }\ni
\varepsilon \rightarrow 0$, 
\begin{eqnarray*}
(II) &\rightarrow &\iiint_{\Omega \times Y\times Z}u_{0}(x,y,z)\varphi
(x-\xi )\psi (y-r,z-s)dxdydz \\
&=&\iiint_{(\Omega -\xi )\times Y\times Z}u_{0}(x+\xi ,y+r,z+s)\varphi
(x)\psi (y,z)dxdydz.
\end{eqnarray*}%
On the other hand, 
\begin{eqnarray*}
\left\vert (I)\right\vert &\leq &C\left\Vert \psi \left( \cdot -\frac{\xi }{%
\varepsilon },\cdot -\frac{\xi }{\varepsilon ^{2}}\right) -\psi \left( \cdot
-r,\cdot -s\right) \right\Vert _{\infty } \\
&=&C\left\Vert \psi \left( \cdot -\mathcal{R}\left( \frac{\xi }{\varepsilon }%
\right) ,\cdot -\mathcal{R}\left( \frac{\xi }{\varepsilon ^{2}}\right)
\right) -\psi \left( \cdot -r,\cdot -s\right) \right\Vert _{\infty }
\end{eqnarray*}%
since $\psi $ is $Y\times Z$-periodic, where $\left\Vert \cdot \right\Vert
_{\infty }$ stands for the supremum norm. The uniform continuity of $\psi $
and the convergence results (\ref{3.1}) yield $(I)\rightarrow 0$ as $%
E^{\prime }\ni \varepsilon \rightarrow 0$. The result follows from the
boundedness of $(v_{\varepsilon })_{\varepsilon \in E}$ in $L^{p}(\Omega
-\xi )$.
\end{proof}

The next result deals with the convergence of convolution of sequences.
Before we can proceed further, let $p,q,m\geq 1$ be real numbers satisfying $%
\frac{1}{p}+\frac{1}{q}=1+\frac{1}{m}$. Let $(u_{\varepsilon })_{\varepsilon
\in E}\subset L^{p}(\Omega )$ and $(v_{\varepsilon })_{\varepsilon \in
E}\subset L^{q}(\mathbb{R}^{N})$ be two sequences. Viewing $u_{\varepsilon }$
as defined on the whole $\mathbb{R}^{N}$ (by taking its zero-extension off $%
\Omega $) we define $u_{\varepsilon }\ast v_{\varepsilon }$ on $\mathbb{R}%
^{N}$ by 
\begin{equation*}
(u_{\varepsilon }\ast v_{\varepsilon })(x)=\int_{\mathbb{R}%
^{N}}u_{\varepsilon }(\xi )v_{\varepsilon }(x-\xi )d\xi \ \ \ \ (x\in 
\mathbb{R}^{N}).
\end{equation*}%
Then $u_{\varepsilon }\ast v_{\varepsilon }\in L^{m}(\mathbb{R}^{N})$ and
further 
\begin{equation*}
\left\Vert u_{\varepsilon }\ast v_{\varepsilon }\right\Vert _{L^{m}(\mathbb{R%
}^{N})}\leq \left\Vert u_{\varepsilon }\right\Vert _{L^{p}(\Omega
)}\left\Vert v_{\varepsilon }\right\Vert _{L^{q}(\mathbb{R}^{N})}.
\end{equation*}%
As in \cite{Visintin} where the double convolution is defined, we define the 
\emph{triple convolution} denoted by $\ast \ast $ as follows. For $u\in
L^{p}(\mathbb{R}^{N};L_{\text{per}}^{p}(Y\times Z))\equiv L^{p}(\mathbb{R}%
^{N}\times Y\times Z)$ and $v\in L^{q}(\mathbb{R}^{N}\times Y\times Z)$, $%
u\ast \ast v$ stands for the function 
\begin{equation*}
(u\ast \ast v)(x,y,z)=\iiint_{\mathbb{R}^{N}\times Y\times Z}u(\xi
,r,s)v(x-\xi ,y-r,z-s)d\xi drds,\ \ (x,y,z)\in \mathbb{R}^{N}\times Y\times
Z.
\end{equation*}%
Then $u\ast \ast v$ is well-defined and belongs to $L^{m}(\mathbb{R}%
^{N}\times Y\times Z)$, and further satisfies 
\begin{equation*}
\left\Vert u\ast \ast v\right\Vert _{L^{m}(\mathbb{R}^{N}\times Y\times
Z)}\leq \left\Vert u\right\Vert _{L^{p}(\mathbb{R}^{N}\times Y\times
Z)}\left\Vert v\right\Vert _{L^{q}(\mathbb{R}^{N}\times Y\times Z)}.
\end{equation*}%
Now, if $u\in L^{p}(\Omega \times Y\times Z)$ and $v\in L^{q}(\Omega \times
Y\times Z)$, we define $u\ast \ast v$ just by viewing $\ u$ and $v$ as
defined in the whole $\mathbb{R}^{N}\times Y\times Z$; just take the
zero-extension off $\Omega $.

Bearing this in mind, the next result is in order.

\begin{theorem}
\label{t3.5}Let $(u_{\varepsilon })_{\varepsilon \in E}$ and $%
(v_{\varepsilon })_{\varepsilon \in E}$ be as above. Suppose that, as $E\ni
\varepsilon \rightarrow 0$, $u_{\varepsilon }\rightarrow u_{0}$ reit. in $%
L^{p}(\Omega )$-weak and $v_{\varepsilon }\rightarrow v_{0}$ reit. in $L^{q}(%
\mathbb{R}^{N})$-strong, where $u_{0}\in L^{p}(\Omega \times Y\times Z)$ and 
$v_{0}\in L^{q}(\Omega \times Y\times Z)$. Then, as $E\ni \varepsilon
\rightarrow 0$, 
\begin{equation*}
u_{\varepsilon }\ast v_{\varepsilon }\rightarrow u_{0}\ast \ast v_{0}\text{
reit. in }L^{p}(\Omega )\text{-weak.}
\end{equation*}
\end{theorem}

\begin{proof}
The proof is very similar to the one of its homologue Theorem 2.6 in \cite%
{M2AS} (see also \cite{NA2014}). Since Theorem 2.6 in \cite{M2AS} involves
almost periodicity and moreover is checked in the two-scale sense, it is
suitable to repeat the proof here in the periodicity and multiscale
frameworks for completeness. First and foremost, it is easy to see that the
sequence $(u_{\varepsilon }\ast v_{\varepsilon })_{\varepsilon \in E}$ is
bounded in $L^{m}(\Omega )$. Now, let $\eta >0$ and let $\psi _{0}\in 
\mathcal{K}(\mathbb{R}^{N};\mathcal{C}_{\text{per}}(Y\times Z))$ (the space
of continuous functions from $\mathbb{R}^{N}$ into $\mathcal{C}_{\text{per}%
}(Y\times Z)$ with compact support in $\mathbb{R}^{N}$) be such that $%
\left\Vert v_{0}-\psi _{0}\right\Vert _{L^{q}(\mathbb{R}^{N}\times Y\times
Z)}\leq \frac{\eta }{2}$. Since $v_{\varepsilon }\rightarrow v_{0}$ reit. in 
$L^{q}(\mathbb{R}^{N})$-strong, we have that $v_{\varepsilon }-\psi
_{0}^{\varepsilon }\rightarrow v_{0}-\psi _{0}$ reit. in $L^{q}(\mathbb{R}%
^{N})$-strong, hence $\left\Vert v_{\varepsilon }-\psi _{0}^{\varepsilon
}\right\Vert _{L^{q}(\mathbb{R}^{N})}\rightarrow \left\Vert v_{0}-\psi
_{0}\right\Vert _{L^{q}(\mathbb{R}^{N}\times Y\times Z)}$ as $E\ni
\varepsilon \rightarrow 0$. So, there is $\alpha >0$ such that 
\begin{equation}
\left\Vert v_{\varepsilon }-\psi _{0}^{\varepsilon }\right\Vert _{L^{q}(%
\mathbb{R}^{N})}\leq \eta \text{ for }E\ni \varepsilon \leq \alpha \text{.}
\label{3.7}
\end{equation}%
If we still denote by $u_{\varepsilon }$ the extension by zero of $%
u_{\varepsilon }$ outside $\Omega $, then we have, for any $f\in \mathcal{K}%
(\Omega ;\mathcal{C}_{\text{per}}(Y\times Z))$,%
\begin{eqnarray*}
\int_{\Omega }(u_{\varepsilon }\ast v_{\varepsilon })(x)f\left( x,\frac{x}{%
\varepsilon },\frac{x}{\varepsilon ^{2}}\right) dx &=&\int_{\Omega }\left(
\int_{\mathbb{R}^{N}}u_{\varepsilon }(t)v_{\varepsilon }(x-t)dt\right)
f\left( x,\frac{x}{\varepsilon },\frac{x}{\varepsilon ^{2}}\right) dx \\
&=&\int_{\mathbb{R}^{N}}u_{\varepsilon }(t)\left[ \int_{\mathbb{R}%
^{N}}v_{\varepsilon }(x-t)f\left( x,\frac{x}{\varepsilon },\frac{x}{%
\varepsilon ^{2}}\right) dx\right] dt \\
&=&\int_{\mathbb{R}^{N}}u_{\varepsilon }(t)\left[ \int_{\mathbb{R}%
^{N}}v_{\varepsilon }(x)f\left( x+t,\frac{x}{\varepsilon }+\frac{t}{%
\varepsilon },\frac{x}{\varepsilon ^{2}}+\frac{t}{\varepsilon ^{2}}\right) dx%
\right] dt \\
&=&\int_{\mathbb{R}^{N}}u_{\varepsilon }(t)\left[ \int_{\mathbb{R}%
^{N}}(v_{\varepsilon }(x)-\psi _{0}^{\varepsilon }(x))f^{\varepsilon }(x+t)dx%
\right] dt \\
&&+\int_{\mathbb{R}^{N}}u_{\varepsilon }(t)\left( \int_{\mathbb{R}^{N}}\psi
_{0}^{\varepsilon }(x)f^{\varepsilon }(x+t)dx\right) dt \\
&=&(I)+(II).
\end{eqnarray*}%
Firstly, $(I)=\int_{\Omega }[u_{\varepsilon }\ast (v_{\varepsilon }-\psi
_{0}^{\varepsilon })](x)f^{\varepsilon }(x)dx$ and 
\begin{eqnarray*}
\left\vert (I)\right\vert &\leq &\left\Vert u_{\varepsilon }\right\Vert
_{L^{p}(\Omega )}\left\Vert v_{\varepsilon }-\psi _{0}^{\varepsilon
}\right\Vert _{L^{q}(\mathbb{R}^{N})}\left\Vert f^{\varepsilon }\right\Vert
_{L^{m^{\prime }}(\Omega )} \\
&\leq &C\left\Vert v_{\varepsilon }-\psi _{0}^{\varepsilon }\right\Vert
_{L^{q}(\mathbb{R}^{N})},
\end{eqnarray*}%
$C>0$ being independent of $\varepsilon $. It follows from (\ref{3.7}) that 
\begin{equation}
\left\vert (I)\right\vert \leq c\eta \text{ for }0<\varepsilon \leq \alpha 
\text{.}  \label{3.8}
\end{equation}%
Secondly, owing to Theorem \ref{t3.4}, we have, as $E\ni \varepsilon
\rightarrow 0$, 
\begin{equation*}
\int_{\mathbb{R}^{N}}\psi _{0}^{\varepsilon }(x)f^{\varepsilon
}(x+t)dx\rightarrow \iint_{\mathbb{R}^{N}\times Y\times Z}\psi
_{0}(x,y,z)f(x+t,y+r,z+s)dxdydz
\end{equation*}%
where $r=\lim \mathcal{R}\left( \frac{t}{\varepsilon }\right) $ and $s=\lim 
\mathcal{R}\left( \frac{t}{\varepsilon ^{2}}\right) $ (for a suitable
subsequence $E^{\prime }$ of $E$). So let $\Phi :\mathbb{R}^{N}\times
Y\times Z\rightarrow \mathbb{R}$ be defined by 
\begin{equation*}
\Phi (t,r,s)=\iint_{\mathbb{R}^{N}\times Y\times Z}\psi
_{0}(x,y,z)f(x+t,y+r,z+s)dxdydz,\ (t,r,s)\in \mathbb{R}^{N}\times Y\times Z.
\end{equation*}%
Then it is easy to see that $\Phi \in \mathcal{K}(\mathbb{R}^{N};\mathcal{C}%
_{\text{per}}(Y\times Z))$, so that the trace $\Phi ^{\varepsilon }(t)=\Phi
\left( t,\frac{t}{\varepsilon },\frac{t}{\varepsilon ^{2}}\right) $ ($t\in 
\mathbb{R}^{N}$) is well-defined by 
\begin{equation*}
\Phi ^{\varepsilon }(t)=\iint_{\mathbb{R}^{N}\times Y\times Z}\psi
_{0}(x,y,z)f\left( x+t,y+\frac{t}{\varepsilon },z+\frac{t}{\varepsilon ^{2}}%
\right) dxdydz.
\end{equation*}

Next, we have 
\begin{eqnarray*}
(II) &=&\int_{\mathbb{R}^{N}}u_{\varepsilon }(t)\left( \int_{\mathbb{R}%
^{N}}\psi _{0}^{\varepsilon }(x)f^{\varepsilon }(x+t)dx-\Phi ^{\varepsilon
}(t)\right) dt+\int_{\mathbb{R}^{N}}u_{\varepsilon }(t)\Phi ^{\varepsilon
}(t)dt \\
&=&(II_{1})+(II_{2}).
\end{eqnarray*}%
Dealing with $(II_{1})$, let 
\begin{equation*}
V_{\varepsilon }(t)=\int_{\mathbb{R}^{N}}\psi _{0}^{\varepsilon
}(x)f^{\varepsilon }(x+t)dx-\Phi ^{\varepsilon }(t)\text{ for a.e. }t\in 
\mathbb{R}^{N}.
\end{equation*}%
Then the following holds true:

\begin{itemize}
\item[(P1)] For a.e. $t$, $V_{\varepsilon }(t)\rightarrow 0$ as $E^{\prime
}\ni \varepsilon \rightarrow 0$ (possibly up to a subsequence)

\item[(P2)] $\int_{\mathbb{R}^{N}}u_{\varepsilon }(t)V_{\varepsilon
}(t)dt\rightarrow 0$ as $E^{\prime }\ni \varepsilon \rightarrow 0$.
\end{itemize}

Indeed, for (P1), appealing to Theorem \ref{t3.4}, 
\begin{equation*}
\int_{\mathbb{R}^{N}}\psi _{0}^{\varepsilon }(x)f^{\varepsilon
}(x+t)dx\rightarrow \iint_{\mathbb{R}^{N}\times Y\times Z}\psi
_{0}(x,y,z)f(x+t,y+r,z+s)dxdydz\text{ as }E^{\prime }\ni \varepsilon
\rightarrow 0
\end{equation*}%
where $r$ and $s$ are such that $\mathcal{R}(\frac{t}{\varepsilon }%
)\rightarrow r$ and $\mathcal{R}(\frac{t}{\varepsilon ^{2}})\rightarrow s$
for some subsequence of $E^{\prime }$ (not relabeled). Moreover, since $\Phi
^{\varepsilon }(t)=\Phi \left( t,\mathcal{R}(\frac{t}{\varepsilon }),%
\mathcal{R}(\frac{t}{\varepsilon ^{2}})\right) $, we have by the continuity
of $\Phi (t,\cdot ,\cdot )$ that, for the same subsequence, 
\begin{equation*}
\Phi ^{\varepsilon }(t)\rightarrow \iint_{\mathbb{R}^{N}\times Y\times
Z}\psi _{0}(x,y,z)f(x+t,y+r,z+s)dxdydz.
\end{equation*}%
Thus (P1) is justified. As for (P2), first and foremost we have 
\begin{equation*}
\left\vert V_{\varepsilon }(t)\right\vert \leq C\text{ for a.e. }t\in 
\mathbb{R}^{N},
\end{equation*}%
$C>0$ being independent of $t$ and $\varepsilon $. Since $f$ and $\psi _{0}$
belong to $\mathcal{K}(\mathbb{R}^{N};\mathcal{C}_{\text{per}}(Y\times Z))$
we have that $f^{\varepsilon }$ and $\psi _{0}^{\varepsilon }$ lie in $%
\mathcal{K}(\mathbb{R}^{N})$ and their supports are contained in a fixed
compact set of $\mathbb{R}^{N}$. Therefore $\psi _{0}^{\varepsilon }\ast
f^{\varepsilon }\in \mathcal{K}(\mathbb{R}^{N})$. As a result, $%
V_{\varepsilon }\in \mathcal{K}(\mathbb{R}^{N})$ and further its support is
contained in a fixed compact set $L\subset \mathbb{R}^{N}$ independent of $%
\varepsilon $.

With this in mind, let $\gamma >0$. We infer from Egorov's theorem that
there exists $D\subset \mathbb{R}^{N}$ such that meas$(\mathbb{R}%
^{N}\backslash D)<\gamma $ and $V_{\varepsilon }$ converges uniformly to $0$
on $D$. We have the following series of inequalities 
\begin{eqnarray*}
\left\vert \int_{\mathbb{R}^{N}}u_{\varepsilon }(t)V_{\varepsilon
}(t)dt\right\vert &\leq &\left\Vert u_{\varepsilon }\right\Vert
_{L^{p}(D)}\left\Vert V_{\varepsilon }\right\Vert _{L^{p^{\prime
}}(D)}+\left\Vert u_{\varepsilon }\right\Vert _{L^{p}(\mathbb{R}%
^{N}\backslash D)}\left\Vert V_{\varepsilon }\right\Vert _{L^{p^{\prime }}(%
\mathbb{R}^{N}\backslash D)} \\
&\leq &C\left\Vert V_{\varepsilon }\right\Vert _{L^{p^{\prime }}(D\cap L)}+C%
\text{meas}(\mathbb{R}^{N}\backslash D) \\
&\leq &C_{1}\text{meas}(L)\sup_{t\in D}\left\vert V_{\varepsilon
}(t)\right\vert +C_{1}\gamma
\end{eqnarray*}%
where $C_{1}>0$ is independent of both $\varepsilon $ and $D$. It emerges
from the uniform continuity of $V_{\varepsilon }$ in $D$ that there exists $%
\alpha _{1}>0$ with $\alpha _{1}\leq \alpha $ such that 
\begin{equation*}
\left\vert \int_{\mathbb{R}^{N}}u_{\varepsilon }(t)V_{\varepsilon
}(t)dt\right\vert \leq C_{2}\gamma \text{ provided }E^{\prime }\ni
\varepsilon \leq \alpha _{1},
\end{equation*}%
where $C_{2}>0$ is independent of $\varepsilon $. This shows (P2). We
concludes that $(II_{1})\rightarrow 0$ as $E^{\prime }\ni \varepsilon
\rightarrow 0$.

As for $(II_{2})$, 
\begin{equation*}
\int_{\Omega }u_{\varepsilon }(t)\Phi ^{\varepsilon }(t)dt\rightarrow
\iiint_{\Omega \times Y\times Z}u_{0}(t,r,s)\Phi (t,r,s)dtdrds,
\end{equation*}%
and 
\begin{eqnarray*}
&&\iiint_{\Omega \times Y\times Z}u_{0}(t,r,s)\Phi (t,r,s)dtdrds \\
&=&\iiint_{\Omega \times Y\times Z}u_{0}(t,r,s)\left[ \iiint_{\mathbb{R}%
^{N}\times Y\times Z}\psi _{0}(x,y,z)f(x+t,y+r,z+s)dxdydz\right] dtdrds \\
&=&\iiint_{\Omega \times Y\times Z}\left[ \iiint_{\Omega \times Y\times
Z}u_{0}(t,r,s)\psi _{0}(x-t,y-r,z-s)dtdrds\right] f(x,y,z)dxdydz \\
&=&\iiint_{\Omega \times Y\times Z}(u_{0}\ast \ast \psi
_{0})(x,y,z)f(x,y,z)dxdydz.
\end{eqnarray*}%
Thus, there is $0<\alpha _{2}\leq \alpha _{1}$ such that 
\begin{equation}
\left\vert \int_{\Omega }(u_{\varepsilon }\ast \psi _{0}^{\varepsilon
})f^{\varepsilon }dx-\iiint_{\Omega \times Y\times Z}(u_{0}\ast \ast \psi
_{0})fdxdydz\right\vert \leq \frac{\eta }{2}\text{ for }E^{\prime }\ni
\varepsilon \leq \alpha _{2}\text{.}  \label{3.9}
\end{equation}%
Now, let $0<\varepsilon \leq \alpha _{2}$ be fixed. The decomposition 
\begin{eqnarray*}
&&\int_{\Omega }(u_{\varepsilon }\ast v_{\varepsilon })f^{\varepsilon
}dx-\iiint_{\Omega \times Y\times Z}(u_{0}\ast \ast v_{0})fdxdydz \\
&=&\int_{\Omega }\left[ u_{\varepsilon }\ast (v_{\varepsilon }-\psi
_{0}^{\varepsilon })\right] f^{\varepsilon }dx+\iiint_{\Omega \times Y\times
Z}\left[ u_{0}\ast \ast (\psi _{0}-v_{0})\right] fdxdydz \\
&&+\int_{\Omega }(u_{\varepsilon }\ast \psi _{0}^{\varepsilon
})f^{\varepsilon }dx-\iiint_{\Omega \times Y\times Z}(u_{0}\ast \ast \psi
_{0})fdxdydz,
\end{eqnarray*}%
associated to (\ref{3.7})-(\ref{3.9}) lead us to 
\begin{equation*}
\left\vert \int_{\Omega }(u_{\varepsilon }\ast v_{\varepsilon
})f^{\varepsilon }dx-\iiint_{\Omega \times Y\times Z}(u_{0}\ast \ast
v_{0})fdxdydz\right\vert \leq C\eta \text{ for }E^{\prime }\ni \varepsilon
\leq \alpha _{2}.
\end{equation*}%
Here $C$ is a positive constant independent of $\varepsilon $. This
concludes the proof.
\end{proof}

In the present work, we will deal with the following time-dependent version
of multiscale convergence It has been for the first time considered by
Holmbom \cite{Holmbom1} (see also \cite{Holmbom2, CMA, CPAA}). It reads as:
A sequence $(u_{\varepsilon })_{\varepsilon >0}\subset L^{p}(Q)$ ($Q=\Omega
\times (0,T)$, $1\leq p<\infty $) is said to weakly multiscale converge
towards $u_{0}\in L^{p}(Q\times Y\times Z\times \mathcal{T})$ if, as $%
\varepsilon \rightarrow 0$, 
\begin{equation*}
\int_{Q}u_{\varepsilon }(x,t)f\left( x,t,\frac{x}{\varepsilon },\frac{x}{%
\varepsilon ^{2}},\frac{t}{\varepsilon }\right) dxdt\rightarrow
\iiint_{Q\times Y\times Z\times \mathcal{T}}(u_{0}f)(x,t,y,z,\tau
)dxdtdydzd\tau
\end{equation*}%
for all $f\in L^{p^{\prime }}(Q;\mathcal{C}_{\text{per}}(Y\times Z\times 
\mathcal{T}))$, where $\mathcal{T}=[0,1)$.

The conclusions of Theorems \ref{t3.2}, \ref{t3.4} and \ref{t3.5} are still
valid mutatis mutandis in the context of time-dependent multiscale
convergence (change e.g. $\Omega $ into $Q$, $W^{1,p}(\Omega )$ into $%
L^{p}(0,T;W^{1,p}(\Omega ))$ etc.).

Let $a\in \mathbb{R}$ and let $(u_{\varepsilon })_{\varepsilon >0}\subset
L^{p}(Q)$ satisfy $u_{\varepsilon }\rightarrow u_{0}$ reit. in $L^{p}(Q)$%
-weak, where $u_{0}\in L^{p}(Q\times Y\times Z\times \mathcal{T})$. Define $%
v_{\varepsilon }(x,t)=u_{\varepsilon }(x,t+a)$ for $(x,t)\in \Omega \times
(-a,T-a)$. Then $v_{\varepsilon }\rightarrow v_{0}$ reit. in $L^{p}(\Omega
\times (-a,T-a))$-weak, where 
\begin{equation*}
v_{0}(x,t,y,z,\tau )=u_{0}(x,t+a,y,z,\tau +r),\ \ (x,t,y,z,\tau )\in \Omega
\times (-a,T-a)\times Y\times Z\times \mathcal{T},
\end{equation*}%
the micro-translation $r\in \mathcal{T}$ being a cluster point of the
sequence $(\mathcal{R}(\frac{a}{\varepsilon }))_{\varepsilon >0}$.

\section{Derivation of the limiting problem}

Our aim in this section is to pass to the limit in (\ref{2.1})-(\ref{2.8})
as $\varepsilon \rightarrow 0$. Before we can do this, we need a few
preliminary results.

\subsection{Preliminaries}

Let the functions $A_{i}$, $B_{i}$ and $\rho _{j}$ be as in Section 1.
Assuming (\textbf{A3}) (see Section 1), we get that 
\begin{equation}
A_{i},B_{i}\in \mathcal{C}(\overline{Q};L_{\text{per}}^{2}(Y\times \mathcal{T%
})^{N^{2}});\ \ \rho _{j}\in \mathcal{C}(\overline{\Omega };\mathcal{C}_{%
\text{per}}(Y))  \label{4.1}
\end{equation}%
\begin{equation}
(\chi _{Y_{2}},\chi _{Z_{2}})\in L_{\text{per}}^{2}(Y)\times L_{\text{per}%
}^{2}(Z)\text{ with }\left\langle \chi _{Y_{2}}\right\rangle >0\text{ and }%
\left\langle \chi _{Z_{2}}\right\rangle >0  \label{4.2}
\end{equation}%
where $\left\langle \chi _{Y_{2}}\right\rangle =\int_{Y_{2}}dy$ and $%
\left\langle \chi _{Z_{2}}\right\rangle =\int_{Z_{2}}dz$. Recalling that $%
\chi _{j}^{\varepsilon }$ ($j=1,2$) stands for the characteristic function
of the open set $\Omega _{j}^{\varepsilon }$, it comes from (\ref{4.2}) that 
$\chi _{j}^{\varepsilon }\rightarrow \chi _{j}$ reit. in $L^{2}(\Omega )$%
-weak, where $\chi _{2}(y,z)=\chi _{Y_{2}}(y)+\left( 1-\chi
_{Y_{2}}(y)\right) \chi _{Z_{2}}(z)$ and $\chi _{1}(y,z)=1-\chi
_{2}(y,z)\equiv \left( 1-\chi _{Y_{2}}(y)\right) \left( 1-\chi
_{Z_{2}}(z)\right) $. This stems from the fact that $\chi _{2}^{\varepsilon
}(x)=\chi _{Y_{2}}\left( \frac{x}{\varepsilon }\right) +\left( 1-\chi
_{Y_{2}}\left( \frac{x}{\varepsilon }\right) \right) \chi _{Z_{2}}\left( 
\frac{x}{\varepsilon ^{2}}\right) $ and $\chi _{1}^{\varepsilon }(x)=1-\chi
_{2}^{\varepsilon }(x)$ for $x\in \Omega $.

The following result holds true; see e.g. \cite{NG1, ACAP} for its proof.

\begin{lemma}
\label{l4.1}Let $(u_{\varepsilon })_{\varepsilon >0}\subset L^{p}(Q)$ be
such that $u_{\varepsilon }\rightarrow u_{0}$ reit. in $L^{p}(Q)$-weak. Then 
$\chi _{j}^{\varepsilon }u_{\varepsilon }\rightarrow \chi _{j}u_{0}$ reit.
in $L^{p}(Q)$-weak.
\end{lemma}

Let $h\in L^{1}(\Omega )$ and let $E$ be an ordinary sequence. For $j=1,2$,
the sequence $(h\rho _{j}^{\varepsilon })_{\varepsilon \in E}$ defined by 
\begin{equation*}
(h\rho _{j}^{\varepsilon })(x)=h(x)\rho _{j}\left( x,\frac{x}{\varepsilon }%
\right) \text{ for }x\in \Omega
\end{equation*}%
lies in $L^{1}(\Omega )$ and multiscale converges weakly in $L^{1}(\Omega )$
towards the function $h\odot \rho _{j}\in L^{1}(\Omega ;\mathcal{C}_{\text{%
per}}(Y))$ defined by 
\begin{equation*}
(h\odot \rho _{j})(x,y)=h(x)\rho _{j}(x,y)\text{ for a.e. }(x,y)\in \Omega
\times Y;
\end{equation*}%
see e.g., \cite[Example 4.1]{Hom1}. Thus, 
\begin{eqnarray*}
&&\int_{\Omega }h(x)\left( \chi _{1}^{\varepsilon }\rho _{1}^{\varepsilon
}+\chi _{2}^{\varepsilon }\rho _{2}^{\varepsilon }\right) dx \\
&\rightarrow &\iiint_{\Omega \times Y\times Z}\left[ \chi _{1}(y,z)\rho
_{1}(x,y)+\chi _{2}(y,z)\rho _{2}(x,y)\right] h(x)dxdydz
\end{eqnarray*}%
so that, as $E\ni \varepsilon \rightarrow 0$, 
\begin{equation*}
\rho ^{\varepsilon }\rightarrow \iint_{Y\times Z}\left[ \chi _{1}(y,z)\rho
_{1}(x,y)+\chi _{2}(y,z)\rho _{2}(x,y)\right] dydz=\rho (x)\text{ in }%
L^{\infty }(\Omega )\text{-weak}\ast
\end{equation*}%
(recall that $\rho ^{\varepsilon }(x)=\chi _{1}(\frac{x}{\varepsilon },\frac{%
x}{\varepsilon ^{2}})\rho _{1}(x,\frac{x}{\varepsilon })+\chi _{2}(\frac{x}{%
\varepsilon },\frac{x}{\varepsilon ^{2}})\rho _{2}(x,\frac{x}{\varepsilon })$%
). But 
\begin{equation*}
\rho (x)\geq \Lambda ^{-1}\iint_{Y\times Z}\left[ \chi _{1}(y,z)+\chi
_{2}(y,z)\right] dydz=\Lambda ^{-1}>0.
\end{equation*}%
Thus Proposition \ref{p2.2} applies and there exist a subsequence $E^{\prime
}$ of $E$ and a function $\boldsymbol{u}_{0}\in L^{2}(Q)^{N}$ such that, as $%
E^{\prime }\ni \varepsilon \rightarrow 0$, 
\begin{equation}
\boldsymbol{u}^{\varepsilon }\rightarrow \boldsymbol{u}_{0}\text{ in }%
L^{2}(Q)^{N}\text{-strong.}  \label{4.3}
\end{equation}%
We can extract by a diagonal process a subsequence of $(\boldsymbol{u}%
^{\varepsilon })_{\varepsilon \in E^{\prime }}$ not relabeled, that weakly
converges to $\boldsymbol{u}_{0}$ in $L^{2}(0,T;H_{0}^{1}(\Omega )^{N})$, so
that $\boldsymbol{u}_{0}\in L^{2}(0,T;H_{0}^{1}(\Omega )^{N})$.

Here below and in the sequel we shall use the following notation. For $%
\boldsymbol{w}=(\boldsymbol{w}_{0},\boldsymbol{w}_{1},\boldsymbol{w}_{2})$
with $\boldsymbol{w}_{0}\in L^{2}(0,T;H_{0}^{1}(\Omega )^{N})$, $\boldsymbol{%
w}_{1}\in L^{2}(Q\times \mathcal{T};W_{\#}^{1,2}(Y)^{N})$ and $\boldsymbol{w}%
_{2}\in L^{2}(Q\times Y\times \mathcal{T};W_{\#}^{1,2}(Z)^{N})$, we set 
\begin{equation*}
\mathbb{D}\boldsymbol{w}=\nabla \boldsymbol{w}_{0}+\nabla _{y}\boldsymbol{w}%
_{1}+\nabla _{z}\boldsymbol{w}_{2}\equiv (\mathbb{D}_{j}\boldsymbol{w}%
)_{1\leq j\leq N}
\end{equation*}%
where $\mathbb{D}_{j}\boldsymbol{w}=(\mathbb{D}_{j}\boldsymbol{w}%
^{k})_{1\leq k\leq N}$ and $\mathbb{D}_{j}\boldsymbol{w}^{k}=\frac{\partial
w_{0}^{k}}{\partial x_{j}}+\frac{\partial w_{1}^{k}}{\partial y_{j}}+\frac{%
\partial w_{2}^{k}}{\partial z_{j}}$ for $\boldsymbol{w}_{l}=(w_{l}^{k})_{1%
\leq k\leq N}$, $l=0,1,2$.

Bearing this in mind, the following preliminary result holds.

\begin{proposition}
\label{p4.1}There exist a subsequence $E^{\prime }$ of $E$ and vector
functions $(\boldsymbol{u}_{1},\boldsymbol{v}_{1})\in L^{2}(Q\times \mathcal{%
T};W_{\#}^{1,2}(Y)^{N})^{2}$, $(\boldsymbol{u}_{2},\boldsymbol{v}_{2})\in
L^{2}(Q\times Y\times \mathcal{T};W_{\#}^{1,2}(Z)^{N})^{2}$\ such that, as $%
E^{\prime }\ni \varepsilon \rightarrow 0$, 
\begin{equation}
\chi _{1}^{\varepsilon }\nabla \boldsymbol{u}_{\varepsilon }\rightarrow \chi
_{1}\mathbb{D}\boldsymbol{u}\text{ reit. in }L^{2}(Q)^{N^{2}}\text{-weak}
\label{4.5}
\end{equation}%
and%
\begin{equation}
\chi _{2}^{\varepsilon }\nabla \boldsymbol{v}_{\varepsilon }\rightarrow \chi
_{2}\mathbb{D}\boldsymbol{v}\text{ reit. in }L^{2}(Q)^{N^{2}}\text{-weak }
\label{4.6}
\end{equation}%
where $\boldsymbol{u}_{0}$ is that function defined by \emph{(\ref{4.3})}, $%
\boldsymbol{u}=(\boldsymbol{u}_{0},\boldsymbol{u}_{1},\boldsymbol{u}_{2})$
and $\boldsymbol{v}=(\boldsymbol{u}_{0},\boldsymbol{v}_{1},\boldsymbol{v}%
_{2})$.
\end{proposition}

\begin{proof}
We deduce from Lemma \ref{l4.1} and the convergence result (\ref{4.3}) that $%
\chi _{1}^{\varepsilon }\boldsymbol{u}_{\varepsilon }=\chi _{1}^{\varepsilon
}\boldsymbol{u}^{\varepsilon }\rightarrow \chi _{1}\boldsymbol{u}_{0}$ reit.
in $L^{2}(Q)^{N}$-weak as $E^{\prime }\ni \varepsilon \rightarrow 0$, $%
E^{\prime }$ being the same as in (\ref{4.3}). Let us start by showing (\ref%
{4.5}). The same arguments will \ suffice to verify (\ref{4.6}). To do this,
let us set $\boldsymbol{u}_{0}=(u_{0}^{i})_{1\leq i\leq N}$ and $\boldsymbol{%
u}_{\varepsilon }=(u_{\varepsilon }^{i})_{1\leq i\leq N}$. Fix $1\leq i\leq
N $ arbitrarily. If we denote by $\widetilde{\nabla u_{\varepsilon }^{i}}$
the zero-extension of $\nabla u_{\varepsilon }^{i}$ on the whole $\Omega $,
then $\widetilde{\nabla u_{\varepsilon }^{i}}=\chi _{1}^{\varepsilon }\nabla
u_{\varepsilon }^{i}$. Thus the sequence $(\widetilde{\nabla u_{\varepsilon
}^{i}})_{\varepsilon \in E^{\prime }}$ is bounded in $L^{2}(Q)^{N}$, and
hence there exist a subsequence of $E^{\prime }$ (that we still denote by $%
E^{\prime }$) and a function $w^{i}\in L^{2}(Q\times Y\times Z\times 
\mathcal{T})^{N}$ such that $\widetilde{\nabla u_{\varepsilon }^{i}}%
\rightarrow w^{i}$ reit. in $L^{2}(Q)^{N}$-weak as $E^{\prime }\ni
\varepsilon \rightarrow 0$. In view of Lemma \ref{l4.1}, we have $\chi
_{1}^{\varepsilon }\nabla u_{\varepsilon }^{i}\rightarrow \chi _{1}w^{i}$
reit. in $L^{2}(Q)^{N}$-weak as $E^{\prime }\ni \varepsilon \rightarrow 0$.
It follows that $w^{i}=\chi _{1}w^{i}$. Now, let $\Phi \in (\mathcal{C}%
_{0}^{\infty }(Q)\otimes \mathcal{C}_{\text{per}}^{\infty }(Y\times Z\times 
\mathcal{T}))^{N}$ be such that $\Div_{y}\Phi =\Div_{z}\Phi =0$ and $\Phi
(x,t,y,z,\tau )=0$ for $y\in Y_{2}$ or $z\in Z_{2}$ and $\Phi \cdot \nu
_{1}=0$. Then 
\begin{equation*}
\int_{Q}\chi _{1}^{\varepsilon }\nabla u_{\varepsilon }^{i}\cdot \Phi
^{\varepsilon }dxdt=-\int_{Q}\chi _{1}^{\varepsilon }u_{\varepsilon }^{i}(%
\Div\Phi )^{\varepsilon }dxdt
\end{equation*}%
and letting $E^{\prime }\ni \varepsilon \rightarrow 0$ yields 
\begin{equation*}
\iiint_{Q\times Y\times Z\times \mathcal{T}}(w^{i}-\chi _{1}\nabla 
\boldsymbol{u}_{0})\cdot \Phi dxdtdydzd\tau =0
\end{equation*}%
for all $\Phi \in (\mathcal{C}_{0}^{\infty }(Q)\otimes \mathcal{C}_{\text{per%
}}^{\infty }(Y\times Z\times \mathcal{T}))^{N}$ satisfying $\Div_{y}\Phi =%
\Div_{z}\Phi =0$ and $\Phi (x,t,y,z,\tau )=0$ for $y\in Y_{2}$ or $z\in
Z_{2} $. Lemma 4.14 of \cite{AB} entails the existence of $u_{1}^{i}\in
L^{2}(Q\times \mathcal{T};W_{\#}^{1,2}(Y))$ and $u_{2}^{i}\in L^{2}(Q\times
Y\times \mathcal{T};W_{\#}^{1,2}(Z))$ such that $w^{i}=(\nabla
u_{0}^{i}+\nabla _{y}u_{1}^{i}+\nabla _{z}u_{2}^{i})\chi _{1}$. Setting $%
\boldsymbol{u}_{1}=(u_{1}^{i})_{1\leq i\leq N}$ and $\boldsymbol{u}%
_{2}=(u_{2}^{i})_{1\leq i\leq N}$, (\ref{4.5}) holds true.
\end{proof}

\subsection{Homogenization result}

Assume that the functions $\boldsymbol{u}_{0},\boldsymbol{u}_{1},\boldsymbol{%
u}_{2},\boldsymbol{v}_{1}$ and $\boldsymbol{v}_{2}$ are as in Proposition %
\ref{p4.1}. Let $\psi _{0}=(\psi _{0}^{k})_{1\leq k\leq N}\in \mathcal{C}%
_{0}^{\infty }(Q)^{N}$, $\psi _{1}=(\psi _{1,k})_{1\leq k\leq N},\phi
_{1}=(\phi _{1,k})_{1\leq k\leq N}\in (\mathcal{C}_{0}^{\infty }(Q)\otimes 
\mathcal{C}_{\text{per}}^{\infty }(Y\times \mathcal{T}))^{N}$ and $\psi
_{2}=(\psi _{2,k})_{1\leq k\leq N},\phi _{2}=(\phi _{2,k})_{1\leq k\leq
N}\in (\mathcal{C}_{0}^{\infty }(Q)\otimes \mathcal{C}_{\text{per}}^{\infty
}(Y\times Z\times \mathcal{T}))^{N}$ be such that $\psi _{1}(x,t,y,\tau )=0$
for $y\in Y_{2}$, $\phi _{1}(x,t,y,\tau )=0$ for $y\in Y\backslash Y_{2}$, $%
\psi _{2}(x,t,y,z,\tau )=0$ for $y\in Y_{2}$ or $z\in Z_{2}$ and $\phi
_{2}(x,t,y,z,\tau )=0$ for $y\in Y\backslash Y_{2}$ and $z\in Z\backslash
Z_{2}$. Set $\Psi =(\psi _{0},\psi _{1},\psi _{2})$ and $\Phi =(\psi
_{0},\phi _{1},\phi _{2})$, and define $\Psi _{\varepsilon }=\psi
_{0}+\varepsilon \psi _{1}^{\varepsilon }+\varepsilon ^{2}\psi
_{2}^{\varepsilon }$ and $\Phi _{\varepsilon }=\psi _{0}+\varepsilon \phi
_{1}^{\varepsilon }+\varepsilon ^{2}\phi _{2}^{\varepsilon }$ by 
\begin{eqnarray*}
\Psi _{\varepsilon }(x,t) &=&\psi _{0}(x,t)+\varepsilon \psi _{1}\left( x,t,%
\frac{x}{\varepsilon },\frac{t}{\varepsilon }\right) +\varepsilon ^{2}\psi
_{2}\left( x,t,\frac{x}{\varepsilon },\frac{x}{\varepsilon ^{2}},\frac{t}{%
\varepsilon }\right) , \\
\Phi _{\varepsilon }(x,t) &=&\psi _{0}(x,t)+\varepsilon \phi _{1}\left( x,t,%
\frac{x}{\varepsilon },\frac{t}{\varepsilon }\right) +\varepsilon ^{2}\phi
_{2}\left( x,t,\frac{x}{\varepsilon },\frac{x}{\varepsilon ^{2}},\frac{t}{%
\varepsilon }\right) \\
\text{for }(x,t) &\in &Q.
\end{eqnarray*}%
Then $\Psi _{\varepsilon }$ and $\Phi _{\varepsilon }\in \mathcal{C}%
_{0}^{\infty }(Q)^{N}$ and further $\gamma _{1}^{\varepsilon }\Psi
_{\varepsilon }=\gamma _{2}^{\varepsilon }\Phi _{\varepsilon }$ on $\Gamma
_{12}^{\varepsilon }$ (we have in fact $\varepsilon \psi _{1}^{\varepsilon
}+\varepsilon ^{2}\psi _{2}^{\varepsilon }\in \mathcal{C}_{0}^{\infty
}(\Omega _{1}^{\varepsilon }\times (0,T))^{N}$ and $\varepsilon \phi
_{1}^{\varepsilon }+\varepsilon ^{2}\phi _{2}^{\varepsilon }\in \mathcal{C}%
_{0}^{\infty }(\Omega _{2}^{\varepsilon }\times (0,T))^{N}$) and $\gamma
_{1}^{\varepsilon }\Psi _{\varepsilon }=0$ on $\partial \Omega
_{1}^{\varepsilon }\cap \partial \Omega $, $\gamma _{2}^{\varepsilon }\Phi
_{\varepsilon }=0$ on $\partial \Omega _{2}^{\varepsilon }\cap \partial
\Omega $. Thus the vector-function $(\Psi _{\varepsilon },\Phi _{\varepsilon
})$ can be taken as a test function in the variational form of (\ref{2.1})-(%
\ref{2.8}): 
\begin{eqnarray}
&&-\int_{Q}\left( \chi _{1}^{\varepsilon }\rho _{1}^{\varepsilon }%
\boldsymbol{u}_{\varepsilon }\cdot \frac{\partial \Psi _{\varepsilon }}{%
\partial t}+\chi _{2}^{\varepsilon }\rho _{2}^{\varepsilon }\boldsymbol{v}%
_{\varepsilon }\cdot \frac{\partial \Phi _{\varepsilon }}{\partial t}\right)
dxdt  \label{4.7} \\
&&\ \ \ +\int_{Q}\chi _{1}^{\varepsilon }(A_{0}^{\varepsilon }\nabla 
\boldsymbol{u}_{\varepsilon }+A_{1}^{\varepsilon }\ast \nabla \boldsymbol{u}%
_{\varepsilon })\cdot \nabla \Psi _{\varepsilon }dxdt  \notag \\
&&\ \ \ \ \ +\int_{Q}\chi _{2}^{\varepsilon }(B_{0}^{\varepsilon }\nabla 
\boldsymbol{v}_{\varepsilon }+B_{1}^{\varepsilon }\ast \nabla \boldsymbol{v}%
_{\varepsilon })\cdot \nabla \Phi _{\varepsilon }dxdt  \notag \\
&&\ \ \ \ \ \ \ \ -\int_{Q}(\chi _{1}^{\varepsilon }p_{\varepsilon }\Div\Psi
_{\varepsilon }+\chi _{2}^{\varepsilon }q_{\varepsilon }\Div\Phi
_{\varepsilon })dxdt  \notag \\
&=&\int_{Q}(\chi _{1}^{\varepsilon }\rho _{1}^{\varepsilon }f_{1}\cdot \Psi
_{1}^{\varepsilon }+\chi _{2}^{\varepsilon }\rho _{2}^{\varepsilon
}f_{2}\cdot \Phi _{\varepsilon })dxdt.  \notag
\end{eqnarray}%
Our aim is to pass to the limit in (\ref{4.7}) as $E^{\prime }\ni
\varepsilon \rightarrow 0$ ($E^{\prime }$ being as in Proposition \ref{p4.1}%
). Considering the first integral term in the left-hand side of (\ref{4.7}),
we have 
\begin{eqnarray*}
&&\int_{Q}\left( \chi _{1}^{\varepsilon }\rho _{1}^{\varepsilon }\boldsymbol{%
u}_{\varepsilon }\cdot \frac{\partial \Psi _{\varepsilon }}{\partial t}+\chi
_{2}^{\varepsilon }\rho _{2}^{\varepsilon }\boldsymbol{v}_{\varepsilon
}\cdot \frac{\partial \Phi _{\varepsilon }}{\partial t}\right) dxdt \\
&\rightarrow &\int_{Q}\left( \iint_{Y\times Z}(\chi _{1}\rho _{1}+\chi
_{2}\rho _{2})dydz\right) \boldsymbol{u}_{0}\cdot \frac{\partial \psi _{0}}{%
\partial t}dxdt.
\end{eqnarray*}%
This stems from (\ref{4.3}) associated to Lemma \ref{l4.1}. Next, it is easy
to see that 
\begin{eqnarray*}
\nabla \Psi _{\varepsilon } &\rightarrow &\mathbb{D}\Psi \text{ reit. in }%
L^{2}(Q)^{N^{2}}\text{-strong,} \\
\nabla \Phi _{\varepsilon } &\rightarrow &\mathbb{D}\Phi \text{ reit. in }%
L^{2}(Q)^{N^{2}}\text{-strong.}
\end{eqnarray*}%
Appealing to (\ref{4.5})-(\ref{4.6}), it follows from a suitable statement
of Corollary \ref{c3.1} that 
\begin{equation*}
\chi _{1}^{\varepsilon }\nabla \boldsymbol{u}_{\varepsilon }\cdot \nabla
\Psi _{\varepsilon }\rightarrow \chi _{1}\mathbb{D}\boldsymbol{u}\cdot 
\mathbb{D}\Psi \text{ reit. in }L^{2}(Q)^{N^{2}}\text{-weak}
\end{equation*}%
and 
\begin{equation*}
\chi _{2}^{\varepsilon }\nabla \boldsymbol{v}_{\varepsilon }\cdot \nabla
\Phi _{\varepsilon }\rightarrow \chi _{2}\mathbb{D}\boldsymbol{v}\cdot 
\mathbb{D}\Phi \text{ reit. in }L^{2}(Q)^{N^{2}}\text{-weak.}
\end{equation*}%
Thus using $A_{0}$ and $B_{0}$ as test functions entails 
\begin{eqnarray*}
&&\int_{Q}(\chi _{1}^{\varepsilon }A_{0}^{\varepsilon }\nabla \boldsymbol{u}%
_{\varepsilon }\cdot \nabla \Psi _{\varepsilon }+\chi _{2}^{\varepsilon
}B_{0}^{\varepsilon }\nabla \boldsymbol{v}_{\varepsilon }\cdot \nabla \Phi
_{\varepsilon })dxdt \\
&\rightarrow &\iiint_{Q\times Y\times Z\times \mathcal{T}}(\chi _{1}A_{0}%
\mathbb{D}\boldsymbol{u}\cdot \mathbb{D}\Psi +\chi _{2}B_{0}\mathbb{D}%
\boldsymbol{v}\cdot \mathbb{D}\Phi )dxdtdydzd\tau .
\end{eqnarray*}

Now, for the terms involving convolution, it is a fact that $%
A_{1}^{\varepsilon }\rightarrow A_{1}$ reit. in $L^{1}(Q)^{N^{2}}$-strong.
Thus, owing to Theorem \ref{t3.5} we get 
\begin{equation*}
A_{1}^{\varepsilon }\ast \chi _{1}^{\varepsilon }\nabla \boldsymbol{u}%
_{\varepsilon }\rightarrow A_{1}\ast \ast \chi _{1}\mathbb{D}\boldsymbol{u}%
\text{ reit. in }L^{2}(Q)^{N^{2}}\text{-weak.}
\end{equation*}%
We also have 
\begin{equation*}
B_{1}^{\varepsilon }\ast \chi _{2}^{\varepsilon }\nabla \boldsymbol{v}%
_{\varepsilon }\rightarrow B_{1}\ast \ast \chi _{2}\mathbb{D}\boldsymbol{v}%
\text{ reit. in }L^{2}(Q)^{N^{2}}\text{-weak.}
\end{equation*}%
Therefore the same reasoning conducted before leads to 
\begin{eqnarray*}
&&\int_{Q}[\chi _{1}^{\varepsilon }(A_{1}^{\varepsilon }\ast \nabla 
\boldsymbol{u}_{\varepsilon })\cdot \nabla \Psi _{\varepsilon }+\chi
_{2}^{\varepsilon }(B_{1}^{\varepsilon }\ast \nabla \boldsymbol{v}%
_{\varepsilon })\cdot \nabla \Phi _{\varepsilon }]dxdt \\
&\rightarrow &\iiint_{Q\times Y\times Z\times \mathcal{T}}[\chi
_{1}(A_{1}\ast \ast \mathbb{D}\boldsymbol{u})\cdot \mathbb{D}\Psi +\chi
_{2}(B_{1}\ast \ast \mathbb{D}\boldsymbol{v})\cdot \mathbb{D}\Phi
]dxdtdydzd\tau .
\end{eqnarray*}%
Now, dealing with the terms with pressure, we have 
\begin{eqnarray*}
&&\int_{Q}(\chi _{1}^{\varepsilon }p_{\varepsilon }\Div\Psi _{\varepsilon
}+\chi _{2}^{\varepsilon }q_{\varepsilon }\Div\Phi _{\varepsilon })dxdt \\
&&\qquad =\int_{Q}\pi _{\varepsilon }\Div\psi _{0}dxdt+\int_{Q}\left( \chi
_{1}^{\varepsilon }p_{\varepsilon }(\Div_{y}\psi _{1})^{\varepsilon }+\chi
_{2}^{\varepsilon }q_{\varepsilon }(\Div_{y}\phi _{1})^{\varepsilon }\right)
dxdt \\
&&\ \ \ \ \ \ \ \ \ \ \ +\int_{Q}\left( \chi _{1}^{\varepsilon
}p_{\varepsilon }(\Div_{z}\psi _{2})^{\varepsilon }+\chi _{2}^{\varepsilon
}q_{\varepsilon }(\Div_{z}\phi _{2})^{\varepsilon }\right) dxdt \\
&&\qquad \quad \ +\varepsilon \int_{Q}\left( \chi _{1}^{\varepsilon
}p_{\varepsilon }(\Div\psi _{1})^{\varepsilon }+\chi _{2}^{\varepsilon
}q_{\varepsilon }(\Div\phi _{1})^{\varepsilon }\right) dxdt \\
&&\ \ \ \ \ \ \ \ \ \ \ \ \ \ +\varepsilon ^{2}\int_{Q}\left( \chi
_{1}^{\varepsilon }p_{\varepsilon }(\Div\psi _{2})^{\varepsilon }+\chi
_{2}^{\varepsilon }q_{\varepsilon }(\Div\phi _{2})^{\varepsilon }\right) dxdt
\end{eqnarray*}%
where $\Div_{y}$ (resp. $\Div_{z}$) stands for the divergence operator in $%
\mathbb{R}_{y}^{N}$ (resp. $\mathbb{R}_{z}^{N}$), and $(\Div_{y}\psi
_{1})^{\varepsilon }$ and $(\Div_{y}\phi _{1})^{\varepsilon }$ are defined
for $(x,t)\in Q$ by 
\begin{equation*}
(\Div_{y}\psi _{1})^{\varepsilon }(x,t)=(\Div_{y}\psi _{1})\left( x,t,\frac{x%
}{\varepsilon },\frac{t}{\varepsilon }\right) \text{, }(\Div_{y}\phi
_{1})^{\varepsilon }(x,t)=(\Div_{y}\phi _{1})\left( x,t,\frac{x}{\varepsilon 
},\frac{t}{\varepsilon }\right) ,
\end{equation*}
and where 
\begin{equation}
\pi _{\varepsilon }=\chi _{1}^{\varepsilon }p_{\varepsilon }+\chi
_{2}^{\varepsilon }q_{\varepsilon }  \label{4.19}
\end{equation}%
is the global pressure. By virtue of (\ref{2.16}), the sequence $(\pi
_{\varepsilon })_{\varepsilon \in E^{\prime }}$ is bounded in $L^{2}(Q)$, so
that there exist $p\in L^{2}(Q)$, $p_{1},q_{1}\in L^{2}(Q\times Y\times
Z\times \mathcal{T})$ such that, up to a subsequence of $E^{\prime }$, $\pi
_{\varepsilon }\rightarrow p$ in $L^{2}(Q)$-weak, $\chi _{1}^{\varepsilon
}p_{\varepsilon }\rightarrow p_{1}$ reit. in $L^{2}(Q)$-weak and $\chi
_{2}^{\varepsilon }q_{\varepsilon }\rightarrow q_{1}$ reit. in $L^{2}(Q)$%
-weak. Thus 
\begin{eqnarray*}
&&\int_{Q}(\chi _{1}^{\varepsilon }p_{\varepsilon }\Div\Psi _{\varepsilon
}+\chi _{2}^{\varepsilon }q_{\varepsilon }\Div\Phi _{\varepsilon })dxdt \\
&\rightarrow &\int_{Q}p\Div\psi _{0}dxdt+\iiint_{Q\times Y\times Z\times 
\mathcal{T}}\left( \chi _{1}p_{1}\Div_{y}\psi _{1}+\chi _{2}q_{1}\Div%
_{y}\phi _{1}\right) dxdtdydzd\tau \\
&&\ \ \ \ \ \ \ \ \ \ \ \ \ \ \ \ \ +\iiint_{Q\times Y\times Z\times 
\mathcal{T}}\left( \chi _{1}p_{1}\Div_{z}\psi _{2}+\chi _{2}q_{1}\Div%
_{z}\phi _{2}\right) dxdtdydzd\tau .
\end{eqnarray*}%
Finally 
\begin{equation*}
\int_{Q}(\chi _{1}^{\varepsilon }\rho _{1}^{\varepsilon }f_{1}\cdot \Psi
_{1}^{\varepsilon }+\chi _{2}^{\varepsilon }\rho _{2}^{\varepsilon
}f_{2}\cdot \Phi _{\varepsilon })dxdt\rightarrow \int_{Q}\left(
\iint_{Y\times Z}(\chi _{1}\rho _{1}f_{1}+\chi _{2}\rho
_{2}f_{2})dydz\right) \cdot \psi _{0}dxdt.
\end{equation*}%
In order to formulate the result that we have just proved, we note that $\Div%
\boldsymbol{u}^{\varepsilon }=0$ in $\Omega $ since $\Div\boldsymbol{u}%
_{\varepsilon }=0$ in $\Omega _{1}^{\varepsilon }$ and $\Div\boldsymbol{v}%
_{\varepsilon }=0$ in $\Omega _{2}^{\varepsilon }$. This implies that $\Div%
\boldsymbol{u}_{0}=0$ in $\Omega $, $\Div_{y}\boldsymbol{u}_{1}=0$ in $%
Y\backslash Y_{2}$, $\Div_{z}\boldsymbol{u}_{2}=0$ if $y\notin Y_{2}$ and $%
z\notin Z_{2}$, $\Div_{y}\boldsymbol{v}_{1}=0$ in $Y_{2}$ and $\Div_{z}%
\boldsymbol{v}_{2}=0$ if $y\in Y_{2}$ or $z\in Z_{2}$. So we set 
\begin{equation*}
\begin{array}{l}
\mathbb{F}^{1}=\{(\boldsymbol{u}_{0},\boldsymbol{u}_{1},v_{1},u_{2},v_{2})%
\in F^{1}:\Div\boldsymbol{u}_{0}=0\text{ in }\Omega \text{, }\Div_{y}%
\boldsymbol{u}_{1}=0\text{ in }Y\backslash Y_{2}\text{, } \\ 
\Div_{y}\boldsymbol{v}_{1}=0\text{ in }Y_{2}\text{, }\Div_{z}\boldsymbol{u}%
_{2}=0\text{ in }Y\times Z\backslash (Y_{2}\times Z_{2})\text{ and }\Div_{z}%
\boldsymbol{v}_{2}=0\text{ in } \\ 
\ \ \ \ \ \ \ \ \ \ \ \ \ \ \ \ \ \ \ \ \ \ \ \ \ \ \ \ \ \ \ \ \ \ \ \ \ \
\ \ \ \ \ \ \ \ \ \ \ \ \ \ \ \ \ \ \ \{(y,z)\in Y\times Z:y\in Y_{2}\text{
or }z\in Z_{2}\}\};%
\end{array}%
\end{equation*}%
\begin{equation}
\rho (x)=\iint_{Y\times Z}[\chi _{1}(y,z)\rho _{1}(x,y)+\chi _{2}(y,z)\rho
_{2}(x,y)]dydz,\ \ x\in \overline{\Omega }  \label{4.00}
\end{equation}%
and 
\begin{equation*}
\mathbf{f}(x,t)=\iint_{Y\times Z}(\chi _{1}(y,z)\rho
_{1}(x,y)f_{1}(x,t)+\chi _{2}(y,z)\rho _{2}(x,y)f_{2}(x,t))dydz,\ \text{a.e. 
}(x,t)\in Q
\end{equation*}%
where 
\begin{equation*}
F^{1}=L^{2}(0,T;H_{0}^{1}(\Omega )^{N})\times \lbrack L^{2}(Q\times \mathcal{%
T};W_{\#}^{1,2}(Y)]^{2}\times \lbrack L^{2}(Q\times Y\times \mathcal{T}%
;W_{\#}^{1,2}(Z)]^{2}
\end{equation*}%
and 
\begin{equation*}
\mathcal{F}^{\infty }=\mathcal{C}_{0}^{\infty }(Q)^{N}\times \lbrack (%
\mathcal{C}_{0}^{\infty }(Q)\otimes \mathcal{C}_{\text{per}}^{\infty
}(Y\times \mathcal{T}))^{N}]^{2}\times \lbrack (\mathcal{C}_{0}^{\infty
}(Q)\otimes \mathcal{C}_{\text{per}}^{\infty }(Y\times Z\times \mathcal{T}%
))^{N}]^{2}
\end{equation*}%
its smooth counterpart. We have that $\rho \in \mathcal{C}(\overline{\Omega }%
)$ and $\mathbf{f}\in L^{2}(Q)^{N}$. We have just proved the following
result.

\begin{theorem}
\label{t4.1}The vector function $(\boldsymbol{u}_{0},\boldsymbol{u}%
_{1},v_{1},u_{2},v_{2})\in \mathbb{F}^{1}$ solves the variational problem 
\begin{equation}
\left\{ 
\begin{array}{l}
-\int_{Q}\rho \boldsymbol{u}_{0}\cdot \frac{\partial \boldsymbol{\psi }_{0}}{%
\partial t}dxdt+\iiint_{Q\times Y\times Z\times \mathcal{T}}\chi _{1}[A_{0}%
\mathbb{D}\boldsymbol{u}+A_{1}\ast \ast \mathbb{D}\boldsymbol{u}]\cdot 
\mathbb{D}\Psi dxdtdydzd\tau \\ 
\\ 
\ \ +\iiint_{Q\times Y\times Z\times \mathcal{T}}\chi _{2}[B_{0}\mathbb{D}%
\boldsymbol{v}+B_{1}\ast \ast \mathbb{D}\boldsymbol{v}]\cdot \mathbb{D}\Phi
dxdtdydzd\tau \\ 
\\ 
-\int_{Q}p\Div\boldsymbol{\psi }_{0}dxdt-\iiint_{Q\times Y\times Z\times 
\mathcal{T}}\left( \chi _{1}p_{1}\Div_{y}\psi _{1}+\chi _{2}q_{1}\Div%
_{y}\phi _{1}\right) dxdtdydzd\tau \\ 
\\ 
-\iiint_{Q\times Y\times Z\times \mathcal{T}}\left( \chi _{1}p_{1}\Div%
_{z}\psi _{2}+\chi _{2}q_{1}\Div_{z}\phi _{2}\right) dxdtdydzd\tau =\int_{Q}%
\mathbf{f}\cdot \psi _{0}dxdt \\ 
\\ 
\text{for all }(\psi _{0},\psi _{1},\phi _{1},\psi _{2},\phi _{2})\in 
\mathcal{F}^{\infty }\text{ with }\psi _{1}=0\text{ for }y\in Y_{2}\text{,}
\\ 
\\ 
\phi _{1}=0\text{ for }y\in Y\backslash Y_{2}\text{, }\psi _{2}=0\text{ for }%
y\in Y_{2}\text{ or }z\in Z_{2}\text{ and }\phi _{2}=0\text{ for }y\notin
Y_{2} \\ 
\\ 
\text{and }z\notin Z_{2}.%
\end{array}%
\right.  \label{4.8}
\end{equation}
\end{theorem}

Our next purpose is to find the equation satisfied by the function $%
\boldsymbol{u}_{0}$. To do so, we need to construct the \emph{effective
homogenized viscosity tensor}. Before we can do that, let us however recall
that the functions $\chi _{1}$ and $\chi _{2}$ are expressed as follows: 
\begin{equation}
\begin{array}{l}
\chi _{2}(y,z)=\chi _{Y_{2}}(y)+(1-\chi _{Y_{2}}(y))\chi _{Z_{2}}(z) \\ 
\\ 
\chi _{1}(y,z)=1-\chi _{2}(y,z)=(1-\chi _{Y_{2}}(y))(1-\chi _{Z_{2}}(z))%
\text{ for }(y,z)\in Y\times Z.%
\end{array}
\label{4.0}
\end{equation}

With this in mind, if in (\ref{4.8}) we choose consecutively the functions $%
(\psi _{0},\psi _{1},\phi _{1},\psi _{2},\phi _{2})\in \mathcal{F}^{\infty }$
such that: 1) $\psi _{1}=\phi _{1}=\psi _{2}=\phi _{2}=0$, 2) $\psi
_{0}=\phi _{1}=\psi _{2}=\phi _{2}=0$, 3) $\psi _{0}=\psi _{1}=\phi
_{1}=\phi _{2}=0$ and 4) $\psi _{0}=\psi _{1}=\phi _{1}=\psi _{2}=0$, then
we get the system consisting of problems (\ref{4.9})-(\ref{4.13}) below: 
\begin{equation}
\begin{array}{l}
-\int_{Q}\rho \boldsymbol{u}_{0}\cdot \frac{\partial \boldsymbol{\psi }_{0}}{%
\partial t}dxdt+\iiint_{Q\times Y\times Z\times \mathcal{T}}\chi _{1}\left(
A_{0}\mathbb{D}\boldsymbol{u}+A_{1}\ast \ast \mathbb{D}\boldsymbol{u}\right)
\cdot \nabla \psi _{0}dxdtdydzd\tau \\ 
\\ 
+\iiint_{Q\times Y\times Z\times \mathcal{T}}\chi _{2}\left( B_{0}\mathbb{D}%
\boldsymbol{v}+B_{1}\ast \ast \mathbb{D}\boldsymbol{v}\right) \cdot \nabla
\phi _{0}dxdtdydzd\tau -\int_{Q}p\Div\boldsymbol{\psi }_{0}dxdt \\ 
\\ 
\ \ \ =\int_{Q}\mathbf{f}\cdot \psi _{0}dxdt\text{ for all }\psi _{0}\in 
\mathcal{C}_{0}^{\infty }(Q)^{N},%
\end{array}
\label{4.9}
\end{equation}%
\medskip 
\begin{equation}
\left\{ 
\begin{array}{l}
\iiint_{Q\times Y\times Z\times \mathcal{T}}\chi _{1}\left( A_{0}\mathbb{D}%
\boldsymbol{u}+A_{1}\ast \ast \mathbb{D}\boldsymbol{u}\right) \cdot \nabla
_{y}\psi _{1}dxdtdydzd\tau \\ 
\\ 
-\iiint_{Q\times Y\times Z\times \mathcal{T}}\chi _{1}p_{1}\Div_{y}\psi
_{1}dxdtdydzd\tau =0 \\ 
\\ 
\text{for all }\psi _{1}\in (\mathcal{C}_{0}^{\infty }(Q)\otimes \mathcal{C}%
_{\text{per}}^{\infty }(Y\times \mathcal{T}))^{N}\text{ with }\psi _{1}=0%
\text{ for }y\in Y_{2},%
\end{array}%
\right.  \label{4.10}
\end{equation}%
\medskip 
\begin{equation}
\left\{ 
\begin{array}{l}
\iiint_{Q\times Y\times Z\times \mathcal{T}}\chi _{1}\left( A_{0}\mathbb{D}%
\boldsymbol{u}+A_{1}\ast \ast \mathbb{D}\boldsymbol{u}\right) \cdot \nabla
_{z}\psi _{2}dxdtdydzd\tau \\ 
\\ 
-\iiint_{Q\times Y\times Z\times \mathcal{T}}\chi _{1}p_{1}\Div_{z}\psi
_{2}dxdtdydzd\tau =0 \\ 
\\ 
\text{for all }\psi _{2}\in (\mathcal{C}_{0}^{\infty }(Q)\otimes \mathcal{C}%
_{\text{per}}^{\infty }(Y\times Z\times \mathcal{T}))^{N}\text{ with }\psi
_{2}=0 \\ 
\\ 
\text{for }y\in Y_{2}\text{ or }z\in Z_{2}\text{,}%
\end{array}%
\right.  \label{4.11}
\end{equation}%
\medskip 
\begin{equation}
\left\{ 
\begin{array}{l}
\iiint_{Q\times Y\times Z\times \mathcal{T}}\chi _{2}\left( B_{0}\mathbb{D}%
\boldsymbol{v}+B_{1}\ast \ast \mathbb{D}\boldsymbol{v}\right) \cdot \nabla
_{y}\phi _{1}dxdtdydzd\tau \\ 
\\ 
-\iiint_{Q\times Y\times Z\times \mathcal{T}}\chi _{2}q_{1}\Div_{y}\phi
_{1}dxdtdydzd\tau =0 \\ 
\\ 
\text{for all }\phi _{1}\in (\mathcal{C}_{0}^{\infty }(Q)\otimes \mathcal{C}%
_{\text{per}}^{\infty }(Y\times \mathcal{T}))^{N}\text{ with }\phi _{1}=0%
\text{ for }y\in Y\backslash Y_{2},%
\end{array}%
\right.  \label{4.12}
\end{equation}%
and 
\begin{equation}
\left\{ 
\begin{array}{l}
\iiint_{Q\times Y\times Z\times \mathcal{T}}\chi _{2}\left( B_{0}\mathbb{D}%
\boldsymbol{u}+B_{1}\ast \ast \mathbb{D}\boldsymbol{u}\right) \cdot \nabla
_{z}\phi _{2}dxdtdydzd\tau \\ 
\\ 
-\iiint_{Q\times Y\times Z\times \mathcal{T}}\chi _{2}q_{1}\Div_{z}\phi
_{2}dxdtdydzd\tau =0 \\ 
\\ 
\text{for all }\phi _{2}\in (\mathcal{C}_{0}^{\infty }(Q)\otimes \mathcal{C}%
_{\text{per}}^{\infty }(Y\times Z\times \mathcal{T}))^{N}\text{ with }\phi
_{2}=0 \\ 
\\ 
\text{for }y\notin Y_{2}\text{ and }z\notin Z_{2}\text{.}%
\end{array}%
\right.  \label{4.13}
\end{equation}%
Conversely, since (\ref{4.9})-(\ref{4.13}) are made of linear equations,
summing them up we get (\ref{4.8}). Thus, (\ref{4.8}) is equivalent to (\ref%
{4.9})-(\ref{4.13}).

This being so, let us observe that the problems (\ref{4.10}) (resp. (\ref%
{4.11})) and (\ref{4.12}) (resp. (\ref{4.13})) are very similar, so that the
analysis that will be made for the couple (\ref{4.10})-(\ref{4.11}), will be
exactly the same for the couple (\ref{4.12})-(\ref{4.13}), and will
therefore be omitted for the latter couple. With this in mind, let us first
and foremost deal with (\ref{4.11}). If in (\ref{4.11}) we choose $\psi
_{2}(x,t,y,z,\tau )=\varphi (x,t)\theta (y)\boldsymbol{w}(z)\chi (\tau )$
with $\varphi \in \mathcal{C}_{0}^{\infty }(Q)$, $\theta \in \mathcal{C}_{%
\text{per}}^{\infty }(Y)$, $\boldsymbol{w}\in \mathcal{C}_{\text{per}%
}^{\infty }(Z)^{N}$, $\chi \in \mathcal{C}_{\text{per}}^{\infty }(\mathcal{T}%
)$ and $\theta =0$ in $Y_{2}$ or $\boldsymbol{w}=0$ in $Z_{2}$, then (\ref%
{4.11}) becomes (owing to (\ref{4.0})) 
\begin{equation*}
\iint_{Y\times Z}(1-\chi _{Y_{2}}(y))(1-\chi _{Z_{2}}(z))[\left( A_{0}%
\mathbb{D}\boldsymbol{u}+A_{1}\ast \ast \mathbb{D}\boldsymbol{u}\right)
\cdot \nabla _{z}\boldsymbol{w}-p_{1}\Div_{z}\boldsymbol{w}]\theta dydz=0
\end{equation*}%
or equivalently, 
\begin{equation}
\int_{Z\backslash Z_{2}}\left( A_{0}\mathbb{D}\boldsymbol{u}+A_{1}\ast \ast 
\mathbb{D}\boldsymbol{u}\right) \cdot \nabla _{z}\boldsymbol{w}%
dz-\int_{Z\backslash Z_{2}}p_{1}\Div_{z}\boldsymbol{w}dz=0\text{ a.e. in }%
Q\times (Y\backslash Y_{2})\times \mathcal{T}.  \label{4.14}
\end{equation}%
This being so, let $\mathcal{V}_{2,\text{div}_{z}}=\{\psi \in \mathcal{C}_{%
\text{per}}^{\infty }(Z)^{N}:\Div_{z}\psi =0$ in $Z\backslash Z_{2}$ and $%
\psi =0$ in $Z_{2}\}$, and define the space $\mathbb{B}_{\text{per}%
}^{1,2}(Z\backslash Z_{2})$ to be the strong closure in $W_{\text{per}%
}^{1,2}(Z)^{N}$ of $\mathcal{V}_{2,\text{div}_{z}}$. Next, let $\xi \in 
\mathbb{R}^{N^{2}}$ and consider the variational cell problem for $%
\boldsymbol{u}_{2}$: 
\begin{equation}
\left\{ 
\begin{array}{l}
\text{Find }u^{2}(\xi )\in \mathbb{B}_{\text{per}}^{1,2}(Z\backslash Z_{2})%
\text{ such that} \\ 
\\ 
\int_{Z\backslash Z_{2}}\left( A_{0}[\xi +\nabla _{z}u^{2}(\xi )]+A_{1}\ast
\ast (\xi +\nabla _{z}u^{2}(\xi ))\right) \cdot \nabla _{z}\boldsymbol{w}dz=0
\\ 
\\ 
\text{for all }\boldsymbol{w}\in \mathcal{V}_{2,\text{div}_{z}}\text{.}%
\end{array}%
\right.  \label{4.15}
\end{equation}%
Then (\ref{4.15}) is the equivalent version of (\ref{4.14}) but with test
functions $\boldsymbol{w}$ taken in $\mathcal{V}_{2,\text{div}_{z}}$. In
view of the properties of the matrices $A_{0}$ and $A_{1}$, if we proceed
exactly as in the proof of Theorem \ref{t2.1}, then we infer from \cite[%
Theorem 3.2]{Orlik} the existence of $u^{2}(\xi )$ solution (\ref{4.15})
which is unique up to an additive constant. On the other hand, if in (\ref%
{4.15}) we choose $\xi =\nabla \boldsymbol{u}_{0}(x,t)+\nabla _{y}%
\boldsymbol{u}_{1}(x,t,y,\tau )$ for a fixed $(x,t,y,\tau )\in Q\times
(Y\backslash Y_{2})\times \mathcal{T}$ and compare the resulting equation
with (\ref{4.14}) (for test functions taken in $\mathcal{V}_{2,\text{div}%
_{z}}$), then we get from the uniqueness argument that $\boldsymbol{u}%
_{2}=u^{2}(\nabla \boldsymbol{u}_{0}+\nabla _{y}\boldsymbol{u}_{1})$, where
the right-hand side of the preceding equality stands for the function $%
(x,t,y,\tau )\mapsto u^{2}(\nabla \boldsymbol{u}_{0}(x,t)+\nabla _{y}%
\boldsymbol{u}_{1}(x,t,y,\tau ))$ from $Q\times (Y\backslash Y_{2})\times 
\mathcal{T}$ into $\mathbb{B}_{\text{per}}^{1,2}(Z\backslash Z_{2})$.

Let us now consider the variational problem for (\ref{4.10}). If we define
the matrices $C_{0}$ and $C_{1}$ by setting (for $\xi \in \mathbb{R}^{N^{2}}$%
) 
\begin{equation*}
C_{0}\xi =\int_{Z}(1-\chi _{Z_{2}}(z))A_{0}(\xi +\nabla _{z}u^{2}(\xi ))dz
\end{equation*}%
and 
\begin{eqnarray*}
C_{1}\xi &=&\int_{Z}(1-\chi _{Z_{2}}(z))(A_{1}\ast \ast (\xi +\nabla
_{z}u^{2}(\xi )))dz \\
&\equiv &A_{1}\ast \ast \int_{Z}(1-\chi _{Z_{2}}(z))(\xi +\nabla
_{z}u^{2}(\xi ))dz
\end{eqnarray*}%
for a.e. $(x,t,y,\tau )\in Q\times \mathbb{R}_{y,\tau }^{N+1}$, then we see
that $\boldsymbol{u}_{1}(x,t,\cdot ,\tau )$ is the solution to the equation 
\begin{equation*}
\left\{ 
\begin{array}{l}
\int_{Y}(1-\chi _{Y_{2}}(y))[C_{0}[\nabla \boldsymbol{u}_{0}+\nabla _{y}%
\boldsymbol{u}_{1}]+C_{1}\ast \ast (\nabla \boldsymbol{u}_{0}+\nabla _{y}%
\boldsymbol{u}_{1}]\cdot \nabla _{y}\boldsymbol{w}dy=0 \\ 
\\ 
\text{for all }\boldsymbol{w}\in \mathcal{V}_{1,\text{div}_{y}}=\{\psi \in 
\mathcal{C}_{\text{per}}^{\infty }(Y)^{N}:\Div_{z}\psi =0\text{ in }%
Y\backslash Y_{2}\text{ and }\psi =0\text{ in }Y_{2}\}.%
\end{array}%
\right.
\end{equation*}%
So, by fixing once again $\xi \in \mathbb{R}^{N^{2}}$, a similar study
conducted for (\ref{4.10}) reveals that the cell problem 
\begin{equation*}
\left\{ 
\begin{array}{l}
\text{Find }u^{1}(\xi )\in \mathbb{B}_{\text{per}}^{1,2}(Y\backslash Y_{2})%
\text{ such that} \\ 
\\ 
\int_{Y\backslash Y_{2}}\left( C_{0}[\xi +\nabla _{y}u^{1}(\xi )]+C_{1}\ast
\ast (\xi +\nabla _{y}u^{1}(\xi ))\right) \cdot \nabla _{y}\boldsymbol{w}dz=0
\\ 
\\ 
\text{for all }\boldsymbol{w}\in \mathcal{V}_{1,\text{div}_{y}}%
\end{array}%
\right.
\end{equation*}%
possesses a unique solution in $\mathbb{B}_{\text{per}}^{1,2}(Y\backslash
Y_{2})$ (the strong closure in $W_{\text{per}}^{1,2}(Y)^{N}$ of $\mathcal{V}%
_{1,\text{div}_{y}}$) up to a constant. One also obtains that $\boldsymbol{u}%
_{1}=u^{1}(\nabla \boldsymbol{u}_{0})$ in $Q\times \mathcal{T}$. We also
set, for $\xi \in \mathbb{R}^{N^{2}}$ and $(x,t)\in Q$, 
\begin{equation}
D_{0}\xi =\int_{Y\times \mathcal{T}}(1-\chi _{Y_{2}}(y))C_{0}[\xi +\nabla
_{y}u^{1}(\xi )]dyd\tau ,  \label{4.16}
\end{equation}%
\begin{equation}
D_{1}\xi =\int_{Y\times \mathcal{T}}(1-\chi _{Y_{2}}(y))C_{1}[\xi +\nabla
_{y}u^{1}(\xi )]dyd\tau .  \label{4.17}
\end{equation}%
Then we easily see that, for $(x,t)\in Q$, 
\begin{equation*}
D_{0}\xi =\iint_{Y\times Z\times \mathcal{T}}\chi _{1}A_{0}[\xi +\nabla
_{y}u^{1}(\xi )]+\nabla _{z}u^{2}(\xi +\nabla _{y}u^{1}(\xi ))dydzd\tau
\end{equation*}%
and 
\begin{equation*}
D_{1}\xi =\iint_{Y\times Z\times \mathcal{T}}\chi _{1}\left\{ A_{1}\ast \ast
(\xi +\nabla _{y}u^{1}(\xi ))+\nabla _{z}u^{2}(\xi +\nabla _{y}u^{1}(\xi
))\right\} dydzd\tau .
\end{equation*}%
Similar arguments used for (\ref{4.12}) and (\ref{4.13}) lead to the
existence of unique $v^{1}(\xi )$ and $v^{2}(\xi )$ (for $\xi \in \mathbb{R}%
^{N^{2}}$), solutions to the cell problems for (\ref{4.12}) and (\ref{4.13})
respectively, so that $\boldsymbol{v}_{2}=v^{2}(\nabla \boldsymbol{v}%
_{0}+\nabla _{y}\boldsymbol{v}_{1})$ and $\boldsymbol{v}_{1}=v^{1}(\nabla 
\boldsymbol{v}_{0})$. We also define the corresponding homogenized matrices 
\begin{equation*}
E_{0}\xi =\iint_{Y\times Z\times \mathcal{T}}\chi _{2}B_{0}[\xi +\nabla
_{y}v^{1}(\xi )]+\nabla _{z}v^{2}(\xi +\nabla _{y}v^{1}(\xi ))dydzd\tau ,\
(x,t)\in Q
\end{equation*}%
and 
\begin{equation*}
E_{1}\xi =\iint_{Y\times Z\times \mathcal{T}}\chi _{2}[B_{1}\ast \ast (\xi
+\nabla _{y}v^{1}(\xi ))+\nabla _{z}v^{2}(\xi +\nabla _{y}v^{1}(\xi
))]dydzd\tau ,\ (x,t)\in Q.
\end{equation*}%
It is worth noticing that the matrices $D_{0}$ and $D_{1}$ (and the same for 
$E_{0}$ and $E_{1}$) are defined by $D_{0}=(d_{ij}^{0})_{1\leq i,j\leq N}$, $%
D_{1}=(d_{ij}^{1})_{1\leq i,j\leq N}$ where the $d_{ij}^{0}$ and $d_{ij}^{1}$
are obtained by choosing in (\ref{4.16}) and (\ref{4.17}) $\xi =(\delta
_{ij})_{1\leq i,j\leq N}$ (the identity matrix), $\delta _{ij}$ being the
Kronecker delta. Finally, set $\mathcal{A}_{0}=D_{0}+E_{0}$ and $\mathcal{A}%
_{1}=D_{1}+E_{1}$, that is, for any $\xi \in \mathbb{R}^{N^{2}}$, $\mathcal{A%
}_{0}\xi =D_{0}\xi +E_{0}\xi $ and $\mathcal{A}_{1}\xi =D_{1}\xi +E_{1}\xi $%
. Set also 
\begin{equation*}
m_{c}=\int_{Y_{2}}dy\text{ and }m_{p}=\int_{Z_{2}}dz.
\end{equation*}%
The positive constants $m_{c}$ and $m_{p}$ are the porosity of the crack and
pore spaces respectively. The function $\rho $ defined by (\ref{4.00}) is
the effective homogenized density while the matrices $\mathcal{A}_{0}$ and $%
\mathcal{A}_{1}$ are the effective homogenized elasticity tensors which
depend continuously on $(x,t)\in Q$ as seen in the next result whose easy
and classical proof is left to the reader.

\begin{proposition}
\label{p4.2}It holds that

\begin{itemize}
\item[(i)] $\mathcal{A}_{i}$ $(i=0,1)$ are symmetric and further $\mathcal{A}%
_{i}\in \mathcal{C}(Q)^{N^{2}}$;

\item[(ii)] $\mathcal{A}_{0}\lambda \cdot \lambda \geq \alpha \left\vert
\lambda \right\vert ^{2}$ for all $(x,t)\in Q$ and all $\lambda \in \mathbb{R%
}^{N}$, where $\alpha $ is the same as in assumption \emph{(\textbf{A1})};

\item[(iii)] $\rho \in \mathcal{C}(\overline{\Omega })$ and further $\Lambda
^{-1}\leq \rho (x)\leq \Lambda $ for all $x\in \overline{\Omega }$.
\end{itemize}
\end{proposition}

\bigskip Now we consider the anisotropic nonlocal Stokes system 
\begin{equation}
\left\{ 
\begin{array}{l}
\rho \frac{\partial \boldsymbol{u}_{0}}{\partial t}-\Div\left( \mathcal{A}%
_{0}\nabla \boldsymbol{u}_{0}+\int_{0}^{t}\mathcal{A}_{1}(x,t-\tau )\nabla 
\boldsymbol{u}_{0}(x,\tau )d\tau \right) +\nabla p=\mathbf{f}\text{ in }Q \\ 
\\ 
\ \ \ \ \ \ \ \ \ \ \ \ \ \ \ \ \ \ \ \ \ \ \ \ \ \ \ \ \ \ \ \ \ \ \ \ \ \
\ \ \ \ \ \Div\boldsymbol{u}_{0}=0\text{ in }Q \\ 
\\ 
\ \ \ \ \ \ \ \ \ \ \ \ \ \ \ \ \ \ \ \ \ \ \ \ \ \ \ \ \ \ \ \ \ \ \ 
\boldsymbol{u}_{0}=0\text{ on }\partial \Omega \times (0,T) \\ 
\\ 
\boldsymbol{u}_{0}(x,0)=(1-m_{c})(1-m_{p})\boldsymbol{u}%
^{0}(x)+(m_{c}+m_{p}(1-m_{c}))\boldsymbol{v}^{0}(x),\ x\in \Omega .%
\end{array}%
\right.  \label{4.18}
\end{equation}%
In (\ref{4.18}) the function $\boldsymbol{u}_{0}$ is the strong limit of the
global velocity field $\boldsymbol{u}^{\varepsilon }=\chi _{1}^{\varepsilon }%
\boldsymbol{u}_{\varepsilon }+\chi _{2}^{\varepsilon }\boldsymbol{v}%
_{\varepsilon }$ while $p$ is the weak limit of the global pressure $\pi
_{\varepsilon }=\chi _{1}^{\varepsilon }p_{\varepsilon }+\chi
_{2}^{\varepsilon }q_{\varepsilon }$. Moreover, since $\int_{\Omega
_{1}^{\varepsilon }}p_{\varepsilon }dx=\int_{\Omega _{2}^{\varepsilon
}}q_{\varepsilon }dx=0$, we have $\int_{\Omega }\pi _{\varepsilon }dx=0$, so
that $\int_{\Omega }pdx=0$.

Now, in view of (i)-(iii) in Proposition \ref{p4.2} and owing to the fact
that $\int_{\Omega }pdx=0$, we can argue as in the proof of Theorem \ref%
{t2.1} to show that Problem (\ref{4.18}) possesses a unique solution $(%
\boldsymbol{u}_{0},p)$ such that $\boldsymbol{u}_{0}\in
L^{2}(0,T;H_{0}^{1}(\Omega )^{N})$ and $p\in L^{2}(0,T;L^{2}(\Omega )/%
\mathbb{R})$ where $L^{2}(\Omega )/\mathbb{R}$ stands for the space of $v\in
L^{2}(\Omega )$ satisfying $\int_{\Omega }vdx=0$. We can therefore state the
main homogenization result.

\begin{theorem}
\label{t4.2}Assume \emph{(\textbf{A1})-(\textbf{A3})} hold. For any $%
\varepsilon >0$, let $\boldsymbol{u}_{\varepsilon }$ (resp. $\boldsymbol{v}%
_{\varepsilon }$), the velocity field of the fluid in $\Omega
_{1}^{\varepsilon }$ (resp. $\Omega _{2}^{\varepsilon }$) be given by the
system \emph{(\ref{2.1})-(\ref{2.8})}. Let $\pi _{\varepsilon }$ be the
global pressure given by \emph{(\ref{4.19})}. There exist $\boldsymbol{u}\in
L^{\infty }(0,T;L^{2}(\Omega )^{N})$ -- the velocity of the fluid in the
skeleton, $\boldsymbol{v}\in L^{\infty }(0,T;L^{2}(\Omega )^{N})$ -- the
velocity of the fluid in the pores and cracks system, and $p\in
L^{2}(0,T;L^{2}(\Omega )/\mathbb{R})$ such that, as $\varepsilon \rightarrow
0$, $\chi _{1}^{\varepsilon }\boldsymbol{u}_{\varepsilon }\rightarrow 
\boldsymbol{u}$ in $L^{2}(Q)^{N}$-weak, $\chi _{2}^{\varepsilon }\boldsymbol{%
v}_{\varepsilon }\rightarrow \boldsymbol{v}$ in $L^{2}(Q)^{N}$-weak and $\pi
_{\varepsilon }\rightarrow p$ in $L^{2}(Q)$-weak. Moreover $\boldsymbol{u}%
=(1-m_{c})(1-m_{p})\boldsymbol{u}_{0}$ and $\boldsymbol{v}=\boldsymbol{v}%
_{c}+\boldsymbol{v}_{p}$ where $\boldsymbol{v}_{c}=m_{c}\boldsymbol{u}_{0}$
is the velocity of the fluid in the crack space and $\boldsymbol{v}%
_{p}=(1-m_{c})m_{p}\boldsymbol{u}_{0}$ is the velocity of the fluid in the
pore space, and $m_{p}$ (resp. $m_{c}$) is the porosity of the pore (resp.
crack) space and $(\boldsymbol{u}_{0},p)$ is the unique solution to Problem 
\emph{(\ref{4.18})}.
\end{theorem}

\begin{proof}
First, if we substitute in (\ref{4.9}) $\boldsymbol{u}_{1}=u^{1}(\nabla 
\boldsymbol{u}_{0})$, $\boldsymbol{u}_{2}=u^{2}(\nabla \boldsymbol{u}%
_{0}+\nabla _{y}\boldsymbol{u}_{1})$, $\boldsymbol{v}_{1}=v^{1}(\nabla 
\boldsymbol{u}_{0})$ and $\boldsymbol{v}_{2}=v^{2}(\nabla \boldsymbol{u}%
_{0}+\nabla _{y}\boldsymbol{v}_{1})$, we get, after mere computations, the
variational formulation of (\ref{4.18}). Moreover, owing to the uniqueness
of the solution to (\ref{4.18}), we infer that the whole sequence $(%
\boldsymbol{u}^{\varepsilon },\pi _{\varepsilon })$ (where $\boldsymbol{u}%
^{\varepsilon }$ is the global velocity field defined by (\ref{2.21}))
converges as $\varepsilon \rightarrow 0$, in the following way: $\boldsymbol{%
u}^{\varepsilon }\rightarrow \boldsymbol{u}_{0}$ in $L^{2}(Q)^{N}$-strong
and $\pi _{\varepsilon }\rightarrow p$ in $L^{2}(Q)$-weak. Second, because
of both Lemma \ref{l4.1} and the convergence result (\ref{4.3}), we have $%
\chi _{1}^{\varepsilon }\boldsymbol{u}_{\varepsilon }=\chi _{1}^{\varepsilon
}\boldsymbol{u}^{\varepsilon }\rightarrow \chi _{1}\boldsymbol{u}_{0}$ reit.
in $L^{2}(Q)^{N}$-weak when $\varepsilon \rightarrow 0$, hence $\chi
_{1}^{\varepsilon }\boldsymbol{u}_{\varepsilon }\rightarrow \left(
\iint_{Y\times Z}\chi _{1}dydz\right) \boldsymbol{u}_{0}$ in $L^{2}(Q)^{N}$%
-weak, and 
\begin{equation*}
\iint_{Y\times Z}\chi _{1}dydz=(1-m_{c})(1-m_{p}).
\end{equation*}%
Also, as $\varepsilon \rightarrow 0$, $\chi _{2}^{\varepsilon }\boldsymbol{v}%
_{\varepsilon }=\chi _{2}^{\varepsilon }\boldsymbol{u}^{\varepsilon
}\rightarrow \chi _{2}\boldsymbol{u}_{0}$ reit. in $L^{2}(Q)^{N}$-weak,
hence $\chi _{2}^{\varepsilon }\boldsymbol{v}_{\varepsilon }\rightarrow
\left( \iint_{Y\times Z}\chi _{2}dydz\right) \boldsymbol{u}_{0}$ in $%
L^{2}(Q)^{N}$-weak, and 
\begin{equation*}
\iint_{Y\times Z}\chi _{2}dydz=m_{c}+(1-m_{c})m_{p}.
\end{equation*}%
We may therefore set $\boldsymbol{u}=(1-m_{c})(1-m_{p})\boldsymbol{u}_{0}$, $%
\boldsymbol{v}=\boldsymbol{v}_{c}+\boldsymbol{v}_{p}$ with $\boldsymbol{v}%
_{c}=m_{c}\boldsymbol{u}_{0}$ and $\boldsymbol{v}_{p}=(1-m_{c})m_{p}%
\boldsymbol{u}_{0}$. The fact that $\boldsymbol{u}$ and $\boldsymbol{v}$
belong to $L^{\infty }(0,T;L^{2}(\Omega )^{N})$ follows from both the
boundedness of the sequences $(\chi _{1}^{\varepsilon }\boldsymbol{u}%
_{\varepsilon })_{\varepsilon >0}$ and $(\chi _{2}^{\varepsilon }\boldsymbol{%
v}_{\varepsilon })_{\varepsilon >0}$ in $L^{\infty }(0,T;L^{2}(\Omega )^{N})$
(see Lemma \ref{l2.1}; see especially (\ref{2.14}) therein) and uniqueness
of the weak limit. This concludes the proof of the theorem.
\end{proof}

\begin{remark}
\label{r4.1}\emph{We see from the statement of Theorem \ref{t4.2} that the
limiting velocity in skeleton as well as in pores and cracks are both
proportional.}
\end{remark}

\end{document}